\documentclass{amsart}
\usepackage{ifpdf}
\usepackage{amsmath, amsthm, amsfonts, amssymb, amscd}
\usepackage[active]{srcltx}

    \ifpdf

      \usepackage[pdftex]{graphicx}  

    \else

      \usepackage[dvips]{graphicx}  

      \newcommand{\href}[2]{#2}

    \fi

\renewcommand{\fill}{\operatorname{Fill}}
\newcommand{\sing}{\operatorname{Sing}}

\newcommand{\ine}{\operatorname{Ine}}
\newcommand{\ess}{\operatorname{Ess}}
\newcommand{\diamup}{\mc{D}}
\newcommand{\rate}{\eta}

\newcommand{\floor}[1]{\left\lfloor{#1}\right\rfloor}

\newcommand{\abs}[1]{\left\lvert{#1}\right\rvert}
\newcommand{\norm}[1]{\left\|{#1}\right\|}

\DeclareMathOperator{\bd}{\partial}
\DeclareMathOperator{\cl}{cl}
\DeclareMathOperator{\diam}{\rm{diam}}
\DeclareMathOperator{\conv}{\rm{Conv}}

\DeclareMathOperator{\fix}{\rm{Fix}}
\DeclareMathOperator{\proj}{\rm{P}}

\newcommand{\mc}{\mathcal}

\newcommand{\ol}{\overline}
\renewcommand{\hat}{\widehat}
\newcommand{\til}{\widetilde}

\newcommand{\R}{\mathbb{R}}\newcommand{\N}{\mathbb{N}}
\newcommand{\Z}{\mathbb{Z}}\newcommand{\Q}{\mathbb{Q}}
\newcommand{\T}{\mathbb{T}}

\renewcommand{\SS}{\mathbb{S}}

\newcommand{\sm}{\setminus}

\newcommand{\id}{\mathrm{Id}}
\newcommand{\deck}{\operatorname{Deck}}

\newcommand{\ie}{i.e.\ }
\newcommand{\eg}{e.g.\ }

\newtheorem{theorem}{Theorem}[section] 
\newtheorem{corollary}[theorem]{Corollary}
\newtheorem{lemma}[theorem]{Lemma}
\newtheorem{proposition}[theorem]{Proposition}
\newtheorem*{proposition*}{Proposition}

\newtheorem{claim}{Claim}
\newtheorem*{theorem*}{Theorem}
\newtheorem*{claim*}{Claim}

\newtheorem{theoremain}{Theorem}

\newtheorem{corollarymain}[theoremain]{Corollary}
\newtheorem{questionmain}[theoremain]{Question}

\theoremstyle{definition}
\newtheorem{definition}[theorem]{Definition}
\newtheorem{notation}[theorem]{Notation}

\theoremstyle{remark}
\newtheorem{remark}[theorem]{Remark}

\title{Strictly toral dynamics}
\author{Andres Koropecki}
\address{Universidade Federal Fluminense, Instituto de Matem\'atica e Estat\'\i stica, Rua M\'ario Santos Braga S/N, 24020-140 Niteroi, RJ, Brasil}
\email{ak@id.uff.br}
\author{Fabio Armando Tal}
\address{Instituto de Matem\'atica e Estat\'\i stica, Universidade de S\~ao Paulo, Rua do Mat\~ao 1010, Cidade Universit\'aria, 05508-090 S\~ao Paulo, SP, Brazil}
\email{fabiotal@ime.usp.br}
\thanks{The first author was partially supported by CNPq-Brasil. The second author was partially supported by FAPESP and CNPq-Brasil}

\begin{document}

\begin{abstract} 
This article deals with nonwandering (\eg area-preserving) homeomorphisms of the torus $\T^2$ which are homotopic to the identity and strictly toral, in the sense that they exhibit dynamical properties that are not present in homeomorphisms of the annulus or the plane. This includes all homeomorphisms which have a rotation set with nonempty interior. We define two types of points: inessential and essential. The set of inessential points $\ine(f)$ is shown to be a disjoint union of periodic topological disks (``elliptic islands''), while the set of essential points $\ess(f)$ is an essential continuum, with typically rich dynamics (the ``chaotic region''). This generalizes and improves a similar description by J\"ager. The key result is boundedness of these ``elliptic islands'', which allows, among other things, to obtain sharp (uniform) bounds of the diffusion rates. We also show that the dynamics in $\ess(f)$ is as rich as in $\T^2$ from the rotational viewpoint, and  we obtain results relating the existence of large invariant topological disks to the abundance of fixed points.
\end{abstract}

\maketitle


\setcounter{tocdepth}{1}
\tableofcontents
\section*{Introduction}
The purpose of this article is to study homeomorphisms of the torus $\T^2$ homotopic to the identity which exhibit dynamical properties that are intrinsic to the torus, in the sense that they cannot be present in a homeomorphism of the annulus or the plane. We call such homeomorphisms \emph{strictly toral} (a precise definition is given after the statement of Theorem \ref{th:essine}), and they include the homeomorphisms which have a rotation set with nonempty interior (in the sense of Misiurewicz and Ziemian \cite{m-z}). We will give a description of the dynamics of such maps in terms of ``elliptic islands'' and a ``chaotic region'' which generalizes the one given by J\"ager \cite{jager-elliptic}, and most importantly, we prove the boundedness of elliptic islands. This allows to obtain sharp bounds of the diffusion rates in the chaotic region and has a number of applications.

To be precise with our terminology, let us make a definition. Let $\pi\colon \R^2\to \T^2= \R^2/\Z^2$ be the universal covering. The homeomorphism $f\colon \T^2\to \T^2$ is \emph{annular} if there is some lift $\hat{f}\colon \R^2\to \R^2$ of $f$ such that the deviations in the direction of some nonzero $v\in \Z^2$ are uniformly bounded:
$$-M\leq \langle \hat{f}^n(x)-x, v\rangle \leq M\quad \text{for all $x\in \R^2$ and $n\in \Z$}.$$
 If $f$ is annular, it is easy to see that there is a finite covering of $\T^2$ such that the lift of $f$ to this covering has an invariant annular set (see for example \cite[Remark 3.10]{jager-linearization}), so that in some sense the dynamics of $f$ in a finite covering is embedded in an annulus. Therefore, in order to be strictly toral, a map $f$ must not be annular, and it seems reasonable to require that no positive power of $f$ be annular as well. However, this is not sufficient to qualify as strictly toral: in \cite{kt-irrotational}, an example is given of a homeomorphism $f$ isotopic to the identity such that no power of $f$ is annular, but $\fix(f)$ is \emph{fully essential}. This means that $\fix(f)$ contains the complement of some disjoint union of open topological disks (in the case of our example, just one disk). Such dynamics does not deserve to be called strictly toral: after removing the fixed points, what remains is dynamics that takes place on the plane. 
We mention however that a lift to $\R^2$ of such example has a trivial rotation set $\{(0,0)\}$ but has unbounded orbits in all directions.

The boundedness properties of the dynamics of lifts to the universal covering has been the subject of many recent works \cite{jager-bmm, jager-elliptic, jager-linearization, kk-spreading, davalos}, especially in the context of pseudo-rotations and in the area-preserving setting, or under aperiodicity conditions. In particular, using the notion of annular homeomorphism introduced here, saying that some power of $f$ is annular is equivalent to saying that $f$ is \emph{rationally bounded} in the sense of \cite{jager-linearization}. 

We also need to introduce the notion of \emph{essential} points, which plays a central role in this article. A point $x\in \T^2$ is essential for $f$ if the orbit of every neighborhood of $x$ is an essential subset of $\T^2$ (see \S\ref{sec:essential-set}). Roughly speaking, this says that $x$ exhibits a weak form of ``rotational recurrence''. 
The set of essential points of $f$ is denoted by $\ess(f)$, and the set of inessential points is $\ine(f)=\T^2\sm \ess(f)$. Both sets are invariant, and $\ine(f)$ is open.
We restrict our attention to nonwandering homeomorphisms (this includes, for instance, the area-preserving homeomorphisms). Recall that $f$ is nonwandering if any open set intersects some forward iterate of itself. In that case, it is easy to see that inessential points are precisely the points that belong to some periodic open topological disk in $\T^2$ (see \S\ref{sec:essential}). Note that this does not necessarily mean that $\ine(f)$ is a disjoint union of periodic topological disks, since there may be overlapping (for instance, $\ine(f)$ could be the whole torus, as is the case with the identity map). Our main theorem implies that in the strictly toral case, $\ine(f)$ is indeed an inessential set, in fact a union of periodic ``bounded'' disks. 

In order to state our first theorem, let us give some additional definitions. If $U\subset \T^2$ is an open topological disk, then $\diamup(U)$ denotes the diameter of any connected component of $\pi^{-1}(U)$, and if $\diamup(U)<\infty$ we say that $U$ is bounded (see \S\ref{sec:essential-set} for more details). We say that $f$ is irrotational if some lift of $f$ to $\R^2$ has a rotation set equal to $\{(0,0)\}$. 

\begin{theoremain}\label{th:essine} Let $f\colon \T^2\to \T^2$ be a nonwandering homeomorphism homotopic to the identity. Then one of the following  holds:
\begin{itemize} 
\item[(1)] There is $k\in \N$ such that $\fix(f^k)$ is fully essential, and $f^k$ is irrotational; 
\item[(2)] There is $k\in \N$ such that $f^k$ is annular; or
\item[(3)] $\ess(f)$ is nonempty, connected, and fully essential, and $\ine(f)$ is the union of a family $\mc{U}$ of pairwise disjoint open disks such that for each $U\in \mc{U}$, $\diamup(U)$ is bounded by a constant that depends only on the period of $U$.
\end{itemize}
\end{theoremain}

Note that case (1) is very restrictive, as it means that the complement of $\fix(f^k)$ is inessential. One can think of this as a \emph{planar} case; \ie the dynamics of $f^k$ can be seen entirely by looking at  a homeomorphism of the plane (since $\T^2\sm \fix(f)$ can be embedded in the plane).  We emphasize that case (1) does not always imply case (2), as shown by the example in \cite{kt-irrotational}. Henceforth, by \emph{strictly toral} nonwandering homeomorphism we will mean one in which neither case (1) or (2) above holds.

Thus, if $f$ is nonwandering and strictly toral, there is a decomposition of the dynamics into a union $\ine(f)$ (possibly empty) of periodic bounded discs which can be regarded as ``elliptic islands'', and a fully essential set $\ess(f)$ which carries the ``rotational'' part of the dynamics. Figure \ref{fig:zaslavsky} shows an example where both sets are nonempty.

It is worth mentioning that the nonwandering hypothesis in Theorem \ref{th:essine} (and also in Theorem \ref{th:bdfix} below) is essential. Indeed, if $f$ is a homeomorphism of $\T^2$ obtained as the time-one map of the suspension flow of a Denjoy example in the circle, then $f^k$ is non-annular and has no fixed points, for any $k\in \N$, but there is an unbounded invariant disk (corresponding to the suspension of the wandering interval).

The main difficulty to prove Theorem \ref{th:essine} is to show that if $f$ is strictly toral, there are no unbounded periodic disks. This is possible thanks to the following theorem, which is a key result of this article.

\begin{theoremain}\label{th:bdfix} If $f\colon \T^2\to \T^2$ is a nonwandering non-annular homeomorphism homotopic to the identity then one and only one of the following properties hold:
\begin{itemize}
\item[(1)] There exists a constant $M$ such that each $f$-invariant open topological disk $U$ satisfies $\diamup(U)<M$; or
\item[(2)] $\fix(f)$ is fully essential and $f$ is irrotational.
\end{itemize}
\end{theoremain}

\subsection*{Applications}

If $f\colon \T^2\to \T^2$ is a homeomorphism homotopic to the identity and $\hat{f}\colon \R^2\to \R^2$ is a lift of $f$, then given an open set $U\subset \T^2$ we may define (as in \cite{jager-elliptic}) the local rotation set on $U$ as the set $\rho(\hat{f}, U)\subset \R^2$ consisting of all possible limits of sequences of the form $(\hat{f}^n(z_i)-z_i)/n_i$,
where $\pi(z_i)\in U$ $n_i\to \infty$ as $i\to \infty$. 

Observe that in particular $\rho(\hat{f})=\rho(\hat{f},\T^2)$ is the classic rotation set of $\hat{f}$ as defined in \cite{m-z}. If $\rho(\hat{f})$ has nonempty interior, this provides a great deal of global information about $f$; for instance there is positive entropy \cite{llibre-mackay}, abundance of periodic orbits \cite{franks-reali} and ergodic measures with all kinds of rotation vectors \cite{m-z2}.

Assume that $\rho(\hat{f})$ has nonempty interior (and therefore $f$ is strictly toral). We may define the \emph{diffusion rate} $\rate(\hat{f},U)$ on an open disk $U$ as the inner radius of the convex hull of $\rho(\hat{f},U)$ (which does not depend on the lift). Roughly speaking, this measures the minimum linear rate of growth of $U$ in all homological directions. In \cite{jager-elliptic}, a set $\mc{C}(f)$ is defined consisting of all points $x\in \T^2$ such that every neighborhood of $x$ has positive diffusion rate. This implies that $\mc{C}(f)$ has (external) sensitive dependence on initial condition, which is why it is regarded as the ``chaotic'' region. It is also shown in \cite{jager-linearization} that every point of the set $\mc{E}(f)=\T^2\sm\mc{C}(f)$ belongs to some periodic topological disk $U$, and $\rho(\hat{f},U)$ is a single point (however, we will not use these facts, as they are a consequence of Theorem \ref{th:chaotic}).

In smooth area-preserving systems, KAM theory implies that periodic disks frequently appear, even in a persistent way, near elliptic periodic points \cite{moser-kam}. Maximal periodic disks are hence commonly referred to as \emph{elliptic islands}; these are the components of $\mc{E}(f)$. There are well documented examples in the physics literature which exhibit a pattern of elliptic islands surrounded by a complementary ``chaotic'' region with rich dynamical properties, often called the \emph{instability zone}. The most well known example is probably the Chirikov-Taylor standard map \cite{x_1}. Another example, which falls under our hypotheses, is the \emph{Zaslavsky web map} given by $f(x,y)=M^4(x,y)$, where $M(x,y) =  (y, -x-K\sin(2\pi y-c))$ (see Figure \ref{fig:zaslavsky} for its phase portrait when $K=0.19,\, c=1.69$). For such map, properties of the instability region like width and rates of diffusion are better understood (\cite{x_3, pekarsky}), but at present no general theory is known.  Similar dynamics also appear in other models of physical relevance (see \cite{harper}).

\begin{figure}[ht]
\centering
\includegraphics[height=6cm]{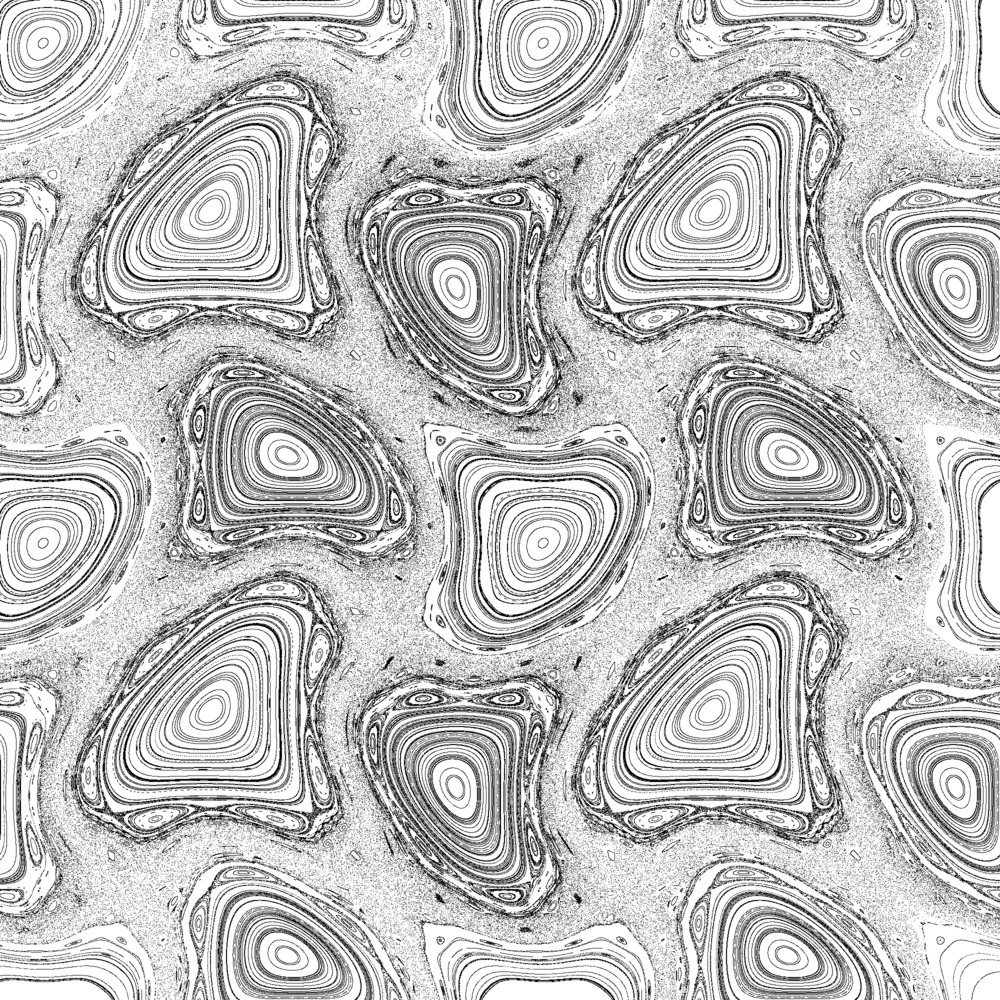}
\caption{Phase portrait for the Zaslavsky web map}
\label{fig:zaslavsky}
\end{figure}

The next result provides a more precise description of the chaotic region in the general setting: it says that $\mc{C}(f)$ concentrates the interesting (from the rotational viewpoint) dynamics, it gives topological information about $\mc{C}(f)$ (namely, it is a fully essential continuum), and it shows that there is \emph{uniform diffusion rate} in the chaotic region. The latter means that is there is a constant $\rate_0>0$ depending only on $f$ such that whenever $x\in \mc{C}(f)$ and $U$ is a neighborhood of $x$, one has $\rate(f,U)=\eta_0$. 

Let us say that an invariant set $\Lambda\subset \T^2$ is \emph{externally transitive} if for any pair of open sets $U$, $V$ intersecting $\Lambda$ there is $n\in \Z$ such that $f^n(U)\cap V\neq \emptyset$, and $\Lambda$ is \emph{externally sensitive on initial conditions} if there is $c>0$ such that for any $x\in \Lambda$ and any neighborhood $U$ of $x$ there is $n\in\N$ such that $\diam(f^n(U))>c$. 
 
\begin{theoremain}\label{th:chaotic} Let $f\colon \T^2\to \T^2$ is a nonwandering homeomorphism homotopic to the identity and $\hat{f}$ a lift of $f$ to $\R^2$. Suppose that $\rho(\hat{f})$ has nonempty interior. Then,
\begin{enumerate}
\item[(1)] $\mc{C}(f)=\ess(f)$, which is a fully essential continuum and $\mc{E}(f)=\ine(f)$ is a disjoint union of periodic bounded disks;
\item[(2)] $\ess(f)$ is externally transitive and sensitive on initial conditions;
\item[(3)] For any $x\in \ess(f)$ and any neighborhood $U$ of $x$, $\conv(\rho(\hat{f},U))=\rho(\hat{f})$;
\end{enumerate}
\end{theoremain}

Another result that reflects how $\ess(f)$ carries rich dynamics is related to the realization of rotation vectors by measures or periodic orbits (see \S\ref{sec:rotation} for definitions). It is known that every extremal or interior point $v$ of $\rho(f)$ is realized by an ergodic measure \cite{m-z,m-z2}, and if $v\in \Q^2$ then it is realized by a periodic point \cite{franks-reali, franks-reali2, m-z} (see also \cite{franks-reali3}). The next theorem guarantees that one can obtain such type of realization \emph{in the set of essential points}.

\begin{theoremain}\label{th:reali} Suppose $f\colon \T^2\to \T^2$ is nonwandering, homotopic to the identity and strictly toral, and let $\hat{f}$ be a lift of $f$ to $\R^2$. Then,
\begin{enumerate}
\item[(1)] Any rational vector of $\rho(\hat{f})$ that is realized by some periodic point is also realized by a periodic point in $\ess(f)$).
\item[(2)] If $\mu$ is an ergodic Borel probability measure $\mu$ with associated rotation vector $\rho_\mu(\hat{f})\notin \Q^2$, then $\mu$ is supported on $\ess(f)$.
\end{enumerate}
\end{theoremain}

Finally, a simple application of Theorem \ref{th:essine} gives a characterization of the obstruction to transitivity for strictly toral maps:

\begin{corollarymain}\label{coro:trans} Let $f\colon \T^2\to \T^2$ be a nonwandering homeomorphism homotopic to the identity, and assume that $f$ is strictly toral. Then $f$ is transitive if and only if there are no bounded periodic disks.
\end{corollarymain}

\subsection*{Questions}

If $\rho(\hat{f}$) has nonempty interior or is a single totally irrational vector, it is easy to conclude that $f$ is strictly toral. If $\rho(\hat{f})$ is a single vector that is neither rational nor totally irrational (for example $\{(a,0)\}$ with $a$ irrational) one can find examples which are annular (e.g. a rigid rotation) and others which are strictly toral (e.g. Furstenberg's example \cite{furstenberg}). We conjecture that strictly toral behavior is not possible when the rotation set is a single rational vector.

\begin{questionmain} Can a rational pseudo-rotation be strictly toral?
\end{questionmain}

In \cite{kt-pseudo}, a partial result is obtained answering the above question negatively with some weak additional hypotheses.

Finally, we do not know if the bound on the size of inessential periodic disks in Theorem \ref{th:essine} is uniform (independent of the period):

\begin{questionmain} Is there a strictly toral nonwandering homeomorphism $f$ homotopic to the identity such that $\sup\{\diamup(U): U \text{ is a connected component of } \ine(f) \} = \infty$?
\end{questionmain}

Let us say a few words about the proofs. In \S\ref{sec:prelim} we introduce most of the notation and terminology, and we prove some basic results. Most of the burden of this article lies in the proof of Theorem \ref{th:bdfix}. For ease of the exposition, we prove Theorems \ref{th:essine}, \ref{th:chaotic}, \ref{th:reali} and Corollary \ref{coro:trans} assuming Theorem \ref{th:bdfix}. This is done in \S\ref{sec:essential}. The proof of Theorem \ref{th:bdfix} relies strongly on the equivariant version of Brouwer's plane translation theorem due to P. Le Calvez \cite{lecalvez-equivariant} and a recent result of O. Jaulent \cite{jaulent} that allows to apply the theorem of Le Calvez in more general contexts. We state these results in \S\ref{sec:brouwer}. Many ideas present here were used in \cite{lecalvez-equivariant} and \cite{lecalvez-hamiltonian} to study Hamiltonian homeomorphisms; in particular what we call \emph{gradient-like Brouwer foliations} (see \S\ref{sec:gradient}). The novelty is that we do not assume that the maps are Hamiltonian, and not even symplectic; and we use the Brouwer foliations in combination with the non-annularity to bound invariant open sets. A key concept that allows to do that is a linking number associated to a simply connected invariant set and a fixed point, which we introduce in \S\ref{sec:linking} together with some applications regarding open invariant sets of maps which have a gradient-like Brouwer foliation. We think these results may be useful by themselves in other contexts. To use these results in the proof of Theorem \ref{th:bdfix}, which is given in \S\ref{sec:bdfix}, we first assume  that there exist arbitrarily large open connected inessential sets in $\T^2\sm \fix(f)$ and that $\fix(f)$ is not fully essential. These two facts allow us to obtain a gradient-like Brouwer foliation, and then we use the results from \S\ref{sec:linking} and some geometric arguments to arrive to a contradiction. 

\subsection*{Acknowledgements}
The authors would like to thank T. J\"ager and the anonymous referees, for the suggestions and corrections that helped improve this paper.

\section{Preliminaries}\label{sec:prelim}

\subsection{Basic notation}

As usual we identify the torus $\T^2$ with the quotient $\R^2/\Z^2$, with quotient projection $\pi\colon \R^2\to \T^2$, which is the universal covering of $\T^2$.

Usually, if $f$ is a homeomorphism of $\T^2$, a lift of $f$ to the universal covering will be usually denoted by $\hat{f}\colon \R^2\to \R^2$. If $f$ is isotopic to the identity, then $\hat{f}$ commutes with the translations $z\mapsto z+v$, $v\in \R^2$, and so $\hat{f}-\id$ is uniformly bounded.

We write $\Z^2_* = \Z^2\sm \{(0,0)\}$. Given $u,v\in \R^2$, their inner product is denoted by $\langle u, v \rangle$,  and $\proj_v\colon \R^2\to \R$ denotes the projection $\proj_v(u) = \langle u, \frac{v}{\norm{v}} \rangle$. If $v=(a,b)$, we use the notation $v^\perp = (-b,a)$

By a \emph{topological disk} we mean an open set homeomorphic to a disk, and similarly a \emph{topological annulus} is an open set homeomorphic to an annulus.

An \emph{arc} on a surface $S$ is a continuous map $\gamma \colon [0,1]\to S$. If the endpoints $\gamma(0)$ and $\gamma(1)$ coincide, we say that $\gamma$ is a loop. We denote by $[\gamma]$ the set $\gamma([0,1])$, and if $\alpha, \beta$ are two arcs such that $\alpha(1)=\beta(0)$, we write $\alpha*\beta$ for the concatenation, \ie the arc defined by $(\alpha*\beta)(t) = \alpha(2t)$ if $t\in [0,1/2]$ and $\beta(2t-1)$ if $t\in (1/2,1]$.
We also use the notation $(\gamma(t))_{t\in [0,1]}$ to describe the arc $\gamma$.

If $\gamma\subset \T^2$ is a loop, we denote by $\gamma^*\in H_1(\T^2,\Z)$ its homology class, which in the case of $\T^2$ coincides with its free homotopy class. The first homology group $H_1(\T^2, \Z)$ can be identified with $\Z^2$ by the isomorphism that maps $v\in \Z^2$ to $\gamma^*$, where $\gamma$ is any arc joining $(0,0)$ to $v$.

If $f\colon \T^2\to \T^2$ is a homeomorphism isotopic\footnote{By a theorem of Epstein \cite{epstein}, this is equivalent to saying that $f$ is homotopic to the identity} to the identity, we denote by $\mc{I}=(f_t)_{t\in [0,1]}$ an isotopy from $f_0 = \id_{\T^2}$ to $f_1=f$ (\ie $t\mapsto f_t$ is an arc in the space of self-homeomorphisms of $\T^2$ joining the identity to $f$). Any such isotopy can be lifted to an isotopy $\hat{\mc{I}}=(\hat{f}_t)_{t\in [0,1]}$ where $\hat{f}_t\colon \R^2$ to $\R^2$ is a lift of $f_t$ for each $t\in [0,1]$ and $\hat{f}_0=\id_{\R^2}$.

\subsection{Rotation set, irrotational homeomorphisms} \label{sec:rotation}
The rotation set of a lift $\hat{f}$ of a homeomorphism $f\colon \T^2\to \T^2$ homotopic to the identity is denoted by $\rho(\hat{f})$ and defined as the set of all limit points of sequences of the form 
$$\left(\frac{\hat{f}^{n_k}(x_k)-x_k}{n_k}\right)_{k\in \N},$$
where $x_k\in \R^2$, $n_k\in \N$ and $n_k\to \infty$ as $k\to \infty$. 

For an $f$-invariant Borel probability measure $\mu$, the rotation vector associated to $\mu$ is defined as $\rho_\mu(\hat{f}) = \int_{\T^2} \phi d\mu$, where $\phi\colon \T^2\to \R^2$ is the map defined by $\phi(x) = \hat{f}(\hat{x})-\hat{x}$ for some $\hat{x}\in \pi^{-1}(x)$ (which is independent of the choice).
The next proposition collects some basic results about rotation vectors. The results are contained in \cite{m-z}.

\begin{proposition}\label{pro:rotation-set} The following properties hold:
\begin{enumerate} 
\item $\rho(\hat{f})$ is compact and convex;
\item $\rho(\hat{f}^n(x)+v) =  n\rho(\hat{f})+v$ for any $n\in \Z$ and $v\in \Z^2$.
\item If $\mu$ is an $f$-ergodic Borel probability measure such that $\rho_\mu(\hat{f})=w$, then for $\mu$-almost every point $x\in \T^2$ and any $\hat{x}\in \pi^{-1}(x)$,  $$\lim_{n\to \infty} \frac{\hat{f}^n(\hat{x})-\hat{x}}{n}=w.$$
\item If $w\in \rho(\hat{f})$ is an extremal point (in the sense of convex sets) then there is an $f$-ergodic Borel probability measure $\mu$ on $\T^2$ such that $\rho_\mu(\hat{f})=w$.
\end{enumerate}
\end{proposition}

When the rotation set consists of a single vector, $f$ is said to be a \emph{pseudo-rotation}, and when this vector is an integer, $f$ is said to be \emph{irrotational}. Thus $f$ is irrotational if there is a lift $\hat{f}$ such that $\rho(\hat{f})=\{0\}$.

If $\mu$ is an ergodic measure and $\rho_\mu(\hat{f})=v$, we say that the rotation vector $v$ is \emph{realized} by $\mu$. If $v=(p_1/q,p_2/q)$ is a rational vector in reduced form (\ie with $p_1,p_2,q$ mutually coprime integers, $q>0$), then  we say that $v$ is \emph{realized by a periodic orbit} if there is $z\in \R^2$ such that 
$$\hat{f}^q(\hat{z})-\hat{z} = (p_1,p_2)$$
for any $\hat{z}\in \pi^{-1}(z)$. Note that this implies that $f^q(z)=z$ and $\lim_{n\to\infty} (\hat{f}^n(\hat{z})-\hat{z})/n=v$.

\subsection{Foliations}
By an \emph{oriented foliation with singularities} $\mc{F}$ on a surface $S$ we mean a closed set $\sing(\mc{F})$, called the set of singularities, together with an oriented topological foliation $\mc{F}'$ of $S\sm \sing(\mc{F})$. Elements of $\mc{F}'$ are oriented one-dimensional manifolds, and we call them \emph{regular leaves} of $\mc{F}$. 

By a theorem of Whitney \cite{whitney, whitney2}, any such $\mc{F}$ can be embedded in a flow; \ie $\mc{F}$ is the set of (oriented) orbits of some topological flow $\phi\colon S\times \R \to S$ (where the singularities of $\mc{F}$ coincide with the set of fixed points of $\phi$). Therefore one may define the $\alpha$-limit and $\omega$-limit of leaves of $\mc{F}$ in the usual way: if $\Gamma$ is a leaf of $\mc{F}$ and $z_0$ is a point of $\Gamma$, then
$$\omega(\Gamma)= \bigcap_{n\geq 0}\ol{\{\phi(z_0, t) : t\geq n\}},\quad \alpha(\Gamma)= \bigcap_{n\leq 0}\ol{\{\phi(z_0, t) : t\leq n\}}.$$

We say that an arc $\gamma$ is \emph{positively transverse} to an oriented foliation with singularities $\mc{F}$ if $[\gamma]$ does  not contain any singularity, and each intersection of $\gamma$ with a leaf of $\mc{F}$ is topologically transverse and ``from left to right''. More precisely: for each $t_0\in [0,1]$ there is a homeomorphism $h$ mapping a neighborhood $U$ of $\gamma(t_0)$ to an open set $V\subset \R^2$ such that $h$ maps the foliation induced by $\mc{F}$ in $U$ to a foliation by vertical lines oriented upwards, and such that the first coordinate of $t\mapsto h(\gamma(t))$ is increasing in a neighborhood of $t_0$. 

\subsection{Essential, inessential, filled, and bounded sets}\label{sec:essential-set}
We say an open subset $U$ of $\T^2$ is \emph{inessential} if every loop in $U$ is homotopically trivial in $\T^2$; otherwise, $U$ is \emph{essential}. An arbitrary set $E$ is called inessential if it has some inessential neighborhood. We also say that $E$ is \emph{fully essential} if $\T^2\sm E$ is inessential.

If $E\subset \T^2$ is open or closed, its \emph{filling} is the union of $E$ with all the inessential connected components of $\T^2\sm E$, and we denote it by $\fill(E)$. If $E=\fill(E)$ we say that $E$ is \emph{filled}. 

A connected open set $A\subset \T^2$ is \emph{annular} if $\fill(A)$ is homeomorphic to an open topological annulus. Note that $\fill(A)$ is necessarily essential in this case. The following facts are easily verified, and we omit the proofs.
\begin{proposition}\label{pro:essential-set} The following properties hold:
\begin{enumerate}
\item If $E\subset \T^2$ is fully essential and either open or closed, then exactly one connected component of $E$ is essential, and in fact it is fully essential.
\item $\fill(E)$ is inessential if so is $E$, fully essential if so is $E$, and neither one if $E$ is neither.
\item An open set $U\subset \T^2$ has an annular component if and only if $U$ is neither inessential nor fully essential.
\item An open connected set $U\subset \T^2$ is fully essential if and only if the map $\iota_U^*\colon H_1(U, \Z) \to H_1(\T^2,\Z)$ induced by the inclusion $\iota_U\colon U\to \T^2$ is surjective.
\item If $E\subset \T^2$ is an open or closed set invariant by a homeomorphism $f\colon \T^2\to\T^2$, then $\fill(E)$ is also $f$-invariant.
\item Suppose $U$ is open and connected and $\hat{U}$ is a connected component of $\pi^{-1}(U)$. Then
\begin{itemize}
\item $U$ is inessential if and only if $\hat{U}\cap (\hat{U}+v)=\emptyset$ for each $v\in \Z^2_*$;
\item $U$ is annular if and only if there is $v\in \Z^2_*$ such that $\hat{U}=\hat{U}+v$ and $\hat{U}\cap (\hat{U}+kv^\perp)=\emptyset$ for all $k\neq 0$ 
\item $U$ is fully essential if and only if $\hat{U}=\hat{U}+v$ for all $v\in \Z^2$.
\end{itemize}
\end{enumerate}
\end{proposition}

Given an arcwise connected set $E\subset \T^2$, let $\hat{E}$ be a connected component of $\pi^{-1}(E)$. We denote by $\diamup(E)$ the diameter of $\hat{E}$, reflecting the fact that this number is independent of the choice of the component $\hat{E}$. If $\diamup(E)<\infty$, we say that $E$ is bounded, and we say that $E$ is unbounded otherwise. If $v\in \R^2_*$, we denote by $\diamup_v(E)$ the diameter of $\proj_v(\hat{E})$, which is also independent of the choice of $\hat{E}$. Let us state a fact for future reference. Its proof is also straightforward.

\begin{proposition}\label{pro:inessential-bound} If $K\subset \T^2$ is closed and inessential, then there is $M>0$ such that $\diamup(C)\leq M$ for each connected component $C$ of $K$. 
\end{proposition}

\subsection{Annular homeomorphisms}

Let $f\colon \T^2\to \T^2$ be a homeomorphism isotopic to the identity. We say that $f$ is \emph{annular} (with direction $v$) if there is some lift $\hat{f}$ of $f$ to $\R^2$ and some $v\in \Z^2_*$ such that $$-M\leq \proj_{v^\perp}(\hat{f}^n(x) - x) \leq M \quad \forall x\in \R^2,\, \forall n\in \Z.$$

The following facts will be frequently used. Their proofs are elementary and will be omitted for the sake of brevity.
\begin{proposition} \label{pro:annular} The following properties hold: 
\begin{enumerate}
\item \label{pro:annular0} If an open set $A\subset \T^2$ is annular, then $\diam(P_v(A))<\infty$ for some $v\in\Z^2_*$. 
\item \label{pro:annular2} If there is an $f$-invariant annular set, then $f$ is annular.
\item \label{pro:annular3} If $f$ is annular with direction $v$, then $\rho(\hat{f}) \subset \R v$ for some lift $\hat{f}$ of $f$.
\item \label{pro:annular4} If $f$ is nonwandering and $f^n$ is non-annular for all $n\in \N$, then any essential $f$-invariant open set is fully essential.
\item \label{pro:annular5} If $f^n$ is annular for some $n\in \N$ and $f$ has a fixed point, then $f$ is annular.
\end{enumerate}
\end{proposition}

\begin{proposition}\label{pro:wall-annular}
Suppose there is a lift $\hat{f}$ of $f$ to $\R^2$ and an open $\hat{f}$-invariant set $V\subset \R^2$ such that $$\proj_{v}^{-1}((-\infty, a))\subset V \subset \proj_v^{-1}((-\infty, b))$$ for some $a<b$. Then $f$ is annular.
\end{proposition}

\subsection{Collapsing a filled inessential set}

The following proposition says that one can collapse the connected components of a filled compact inessential invariant set to points, while preserving the dynamics outside the given set. It will be convenient later on to simplify the sets of fixed points.

\begin{proposition}\label{pro:collapse} Let $K\subset \T^2$ be a compact inessential filled set, and $f\colon \T^2\to \T^2$ a homeomorphism such that $f(K)=K$. Then there is a continuous surjection $h\colon \T^2\to \T^2$ and a homeomorphism $f'\colon \T^2 \to \T^2$ such that
\begin{itemize}
\item $h$ is homotopic to the identity;
\item $hf = f'h$;
\item $K' = h(K)$ is totally disconnected;
\item $h|_{\T^2\sm K}\colon \T^2\sm K \to \T^2\sm K'$ is a homeomorphism.
\end{itemize}
\end{proposition}

\begin{proof}
Each connected component of $K$ is filled and inessential, so it is a \emph{cellular} continuum (\ie it is an intersection of a nested sequence of closed topological disks). Let $\mc{P}$ be the partition of $\T^2$ into compact sets consisting of all connected components of $K$ together with all sets of the form $\{x\}$ with $x\in \T^2\sm K$. Then $\mc{P}$ is an \emph{upper semicontinuous decomposition}: if $P\in \mc{P}$ and $U$ is a neighborhood of $P$, then there is a smaller neighborhood $V\subset U$ of $P$ such that every element of $\mc{P}$ that intersects $V$ is contained in $U$. This is a direct consequence of the fact that the Hausdorff limit of any sequence of connected components of $K$ must be contained in a connected component of $K$.

An improved version of a theorem of Moore, found in \cite{daverman} (Theorems 13.4 and 25.1) says that for such a decomposition (an upper semicontinuous decomposition of a manifold into cellular sets) one can find a homotopy from $(p_t)_{t\in [0,1]}$ from $\id_{\T^2}$ to a closed surjection $p_1\colon \T^2\to \T^2$ such that $\mc{P} = \{p_1^{-1}(x) : x\in \T^2\}$. This implies that $h=p_1$ is homotopic to the identity, $h(K)$ is totally disconnected and $h|_{\T^2\sm K}$ is a homeomorphism onto $\T^2\sm h(K)$. The map $f'$ is well-defined by the equation $f'h=hf$ because $f$ permutes components of $K$, and it follows easily that $f'$ is a homeomorphism, completing the proof.
\end{proof}

\subsection{Other results}

Let us state for future reference two well-known results. The first one is a version of the classical Brouwer's Lemma; see for example Corollary 2.4 of \cite{fathi}.
\begin{proposition}\label{pro:brouwer-trivial} If an orientation-perserving homeomorphism $f\colon \R^2\to \R^2$ has a  nonwandering point, then $f$ has a fixed point. 
\end{proposition}

The second result is due to Brown and Kister:
\begin{theorem}[\cite{brown-invariant}]\label{th:brown-kister} Suppose $S$ is a (not necessarily compact) oriented surface and $f\colon S\to S$ an orientation-preserving homeomorphism. Then each connected component of $S\sm \fix(f)$ is invariant.
\end{theorem}

\section{Theorem \ref{th:essine} and applications}
\label{sec:essential}

As usual, in this section $f$ denotes a homeomorphism of $\T^2$ homotopic to the identity.

We say that $x\in \T^2$ is an \emph{essential} point if $\bigcup_{k\in \Z} f^k(U)$ is essential for each neighborhood $U$ of $x$. If $x$ is not essential, we say that $x$ is \emph{inessential}. It follows from the definition that:
\begin{itemize}
\item The set $\ine(f)$ of all inessential points is open;
\item The set $\ess(f)$ of all essential points is therefore closed;
\item Both sets are $f$-invariant.
\end{itemize}

\begin{remark} Note that $\ine(f)$ coincides with the union of all inessential open invariant sets. This does not necessarily mean that $\ine(f)$ is inessential: a trivial example would be the identity. One can think of less trivial examples where $\ine(f)$ is essential, but they all seem to have some power with a very large fixed point set (namely, a fully inessential set of periodic points). Theorem \ref{th:essine} says that this is the only possibility, under the assumption that $f$ is non-wandering and $f^n$ is non-annular for all $n\in \N$.
\end{remark}

\subsection{Proof of Theorem \ref{th:essine} (assuming Theorem \ref{th:bdfix})}

We will use Theorem \ref{th:bdfix}, the proof of which is postponed to the next sections.

First note that if $\fix(f^k)$ is essential for some $k$, then Theorem \ref{th:bdfix} applied to $f^k$ implies that either $f^k$ is annular, or $\fix(f^k)$ is fully essential and $f^k$ is irrotational. Thus to prove the theorem it suffices to consider $f$ such that
\begin{itemize}
\item $f^k$ is non-annular, and
\item $\fix(f^k)$ is inessential
\end{itemize}
for all $k\in \N$. We will show under these hypotheses that case $(3)$ holds.

\setcounter{claim}{0}
\begin{claim} Each $x\in \ine(f)$ is contained in a bounded periodic topological disk.
\end{claim}
\begin{proof}
If $\epsilon>0$ is small enough, $U_\epsilon = \bigcup_{k\in \Z} f^k(B_\epsilon(x))$ is inessential and $f$-invariant. Let $D_\epsilon$ be the connected component of $U_\epsilon$ containing $x$. Since $f$ is nonwandering and the components of $U_\epsilon$ are permuted by $f$, there is $k\geq 1$ such that $f^k(D_\epsilon)=D_\epsilon$ and $f^n(D_\epsilon)\cap D_\epsilon=\emptyset$ if $1\leq n<k$. Then $U=\fill(D_\epsilon)$ is a periodic open disk. 
The fact that $U$ is bounded follows from Theorem \ref{th:bdfix} applied to $f^k$ (using the assumption that $\fix(f^k)$ is inessential).
\end{proof}

\begin{claim} $\ess(f)$ is fully essential.
\end{claim}
\begin{proof}
Suppose not. Then $\ine(f)$ is essential and open, and in particular $\ine(f)$ contains some essential loop $\gamma$. By the previous claim and by compactness, there exist finitely many simply connected periodic bounded sets $U_1,\dots, U_j$ such that $[\gamma]\subset U_1\cup \cdots \cup U_j$ (and we may assume that each $U_i$ intersects $[\gamma]$). Thus we may find $M>0$ and $m\in \N$ such that $\diamup(U_i)\leq M$ and $f^m(U_i)=U_i$ for $1\leq i\leq j$. Let $g=f^m$, and choose a lift $\hat{g}\colon \R^2\to \R^2$ of $g$. For each $i$, choose a connected component $\hat{U}_i$ of $\pi^{-1}(U_i)$. Then there is $v_i\in \Z^2$ such that $\hat{g}(\hat{U}_i)=\hat{U}_i+v_i$, and so $\hat{g}^n(\hat{U}_i) = \hat{U}_i+nv_i$ for $n\in \Z$. Since $\diam(\hat{U}_i)\leq M$, this implies that if $x\in U_i$ and $\hat{x}\in \pi^{-1}(x)$, then $\norm{\hat{g}^n(\hat{x})-\hat{x}-nv}\leq M$ (note that this does not depend on the choice of $\hat{x}$, so we may use $\hat{x}\in \hat{U}_i$). If we define $\rho_x \doteq \lim_{n\to \infty} (\hat{g}^n(\hat{x})-\hat{x})/n$, it follows that $\rho_x = v_i$ and this vector depends only on $x$ and the choice of the lift $\hat{g}$. Since this works for any $x\in U_i$, it follows that the map $U_1\cup\cdots \cup U_j\to \Z^2$ defined by $x\mapsto \rho_x$ is locally constant. Since $U_1\cup \cdots \cup U_j$ is connected (because it contains $[\gamma]$ and every $U_i$ intersects $\gamma$) it follows that $\rho_x$ is constant on that set. Therefore $v_1=v_2=\cdots = v_j$, \ie  there is $v\in \Z^2$ such that $\hat{g}(\hat{U}_i) = \hat{U}_i+v$ for $1\leq i\leq j$. Moreover, replacing $\hat{g}$ by a suitable lift of $g$ we may assume that $v=0$. 

Therefore we may assume that $\hat{g}(\hat{U}_i) = \hat{U}_i$ for $1\leq i\leq j$. Thus, if $x\in [\gamma]$ and $\hat{x}\in \pi^{-1}(x)$, then $\norm{\hat{g}^n(\hat{x})-\hat{x}} \leq \max\{\diam(U_i) : 1\leq i\leq j\} \leq M$ for each $n\in \Z$. Let us show that this implies that $g=f^m$ is annular, contradicting our hypothesis: Since $\gamma$ is an essential  loop, it lifts to $\R^2$ to a simple arc $\hat{\gamma}$ joining a point $x\in \R^2$ to $x+w$, for some $w\in \Z^2_*$. Let $\Gamma = \bigcup_{k\in \Z} [\hat{\gamma}]+kw$. Then $\proj_{w^\perp}(\Gamma)\subset [a,b]$ for some $a,b\in \R$, and since $\Gamma\subset \pi^{-1}([\gamma])$ we also have that $\norm{\smash{\hat{g}^n(x)-x}}\leq M$ for each $x\in \Gamma$. If $V_0$ is the connected component of $\R^2\sm \Gamma$ such that $\proj_{w^\perp}^{-1}((-\infty, a))\subset V_0$, then $V_0\subset \proj_{w^\perp}^{-1}((-\infty, b))$, and so 
$$\proj_{w^\perp}^{-1}((-\infty,a-M))\subset \hat{g}^k(V_0)\subset \proj_{w^\perp}^{-1}((-\infty, b+M))$$
for each $k\in \Z$.  Thus, letting $V=\bigcup_{k\in \Z} \hat{g}^k(V_0)$, we have 
$$\proj_{w^\perp}^{-1}((-\infty,a-M))\subset V \subset \proj_{w^\perp}^{-1}((-\infty,b+M)),$$
and $V$ is $\hat{g}$-invariant. By Proposition \ref{pro:wall-annular}, we conclude that $g$ is annular, which is the sought contradiction.
\end{proof}

\begin{claim} Each component of $\ine(f)$ is a periodic topological open disk.
\end{claim}
\begin{proof}
Since $\ine(f)$ is inessential, if $U$ is a connected component of $\ine(f)$ then $U$ is inessential and $f^k$-invariant for some $k$ (because $f$ is nonwandering and $\ine(f)$ is invariant). It follows that $\fill(U)$ is open, inessential, filled  (thus a topological disk) and $f^k$-invariant. Thus $\fill(U)\subset \ine(f)$, and since it is connected and intersects $U$, it follows that $\fill(U)=U$, proving the claim.
\end{proof}

\begin{claim} For each $k\in \N$ there is $M_k$ such that every connected component $U$ of $\ine(f)$ such that $f^k(U)=U$ satisfies $\diamup(U)<M_k$.
\end{claim}
\begin{proof} This is a direct application of Theorem \ref{th:bdfix} to $f^k$, since we are under the assumption that $f^k$ is non-annular and $\fix(f^k)$ is inessential.
\end{proof}

This last claim concludes the proof of Theorem \ref{th:essine}. \qed

\begin{corollary} \label{coro:fully} If $f\colon \T^2\to \T^2$ is homotopic to the identity, nonwandering and strictly toral (\ie cases (1) and (2) of Theorem \ref{th:essine} do not hold), then
\begin{itemize}
\item for any essential point $x$, if $U$ is a neighborhood of $x$ then the set $U'=\bigcup_{n\in \Z} f^n(U)$ is connected and fully essential;
\item $\ess(f^k) = \ess(f)$ for all $k\in \N$.
\end{itemize}
\end{corollary}

\begin{proof} It follows from Proposition \ref{pro:annular}(\ref{pro:annular4}) that $U'$ is fully essential. Since the connected components of $U'$ are permuted by $f'$, they are all homeomorphic to each other, and since one of them is fully essential, all of them must be. But two fully essential open sets cannot be disjoint, so there is only one component, as claimed.

For the second claim note that $\ess(f^k)\subset \ess(f)$ follows directly from the definition. On the other hand if $x\notin \ess(f^k)$ then $x\in \ine(f^k)$. This means that if $U$ is a small enough neighborhood of $x$, then $U'=\bigcup_{n\in \Z} f^{kn}(U)$ is inessential, so $U'\subset \ine(f^k)$. On the other hand, if $i\in \Z$ then the $f^k$-orbit of $f^i(U)$ is $f^i(U')$, which is also inessential, so $f^i(U')\subset \ine(f^k)$. This implies that $U'':=\bigcup_{i\in \Z} f^i(U') = \bigcup_{n\in \Z} f^n(U)\subset \ine(f^k)$. But Theorem \ref{th:essine} applied to $f^k$ implies that $\ine(f^k)$ is inessential, so that $U''$ is inessential as well, and we conclude that $x\in \ine(f)$. Therefore, $\ine(f^k)\subset \ine(f)$, and so $\ess(f^k)\supset \ess(f)$, completing the proof.
\end{proof}

\subsection{Proof of Theorem \ref{th:chaotic}}
\label{sec:chaotic}

Assume that $f\colon \T^2\to \T^2$ is a nonwandering homeomorphism homotopic to the identity, $\hat{f}$ is a lift of $f$ to $\R^2$, and $\rho(\hat{f})$ has nonempty interior. This implies that $f$ is stricly toral, so that only case (3) in Theorem \ref{th:essine} holds. 

Recall that $\mc{E}(f)$ as the set of all $x\in \T^2$ such that $\rho(\hat{f},U)$ is a single vector of $\Q^2$ for some neighborhood $U$ of $x$, and $\mc{C}(f) = \T^2\sm \mc{E}(f)$. 

We want to show that
\begin{enumerate}
\item[(1)] $\mc{C}(f)=\ess(f)$, which is a fully essential continuum and $\mc{E}(f)=\ine(f)$ is a disjoint union of periodic bounded disks;
\item[(2)] $\ess(f)$ is externally transitive and sensitive on initial conditions;
\item[(3)] For any $x\in \ess(f)$ and any neighborhood $U$ of $x$, $\conv(\rho(\hat{f},U))=\rho(\hat{f})$;
\end{enumerate}

Let us begin with the following
\begin{claim*} For any $x\in \ess(f)$ and any neighborhood $U$ of $x$, $\conv(\rho(\hat{f},U))=\rho(\hat{f})$. 
\end{claim*}
\begin{proof}
 Recall from \cite{m-z} that $\rho(\hat{f})$ is convex, and if $v\in \rho(\hat{f})$ is extremal (in the sense of convexity) then there is at least one point $z\in \R^2$ such that $(\hat{f}^n(z)-z)/n\to v$ as $n\to \infty$. Let $x\in \ess(f)$, and suppose for contradiction that $\conv(\rho(\hat{f},U))\neq \rho(\hat{f})$ for some neighborhood $U$ of $x$. Since the two sets are convex and compact, and $\conv(\rho(\hat{f},U))\subset \rho(\hat{f})$, this implies that there is a direction $w\in \R^2_*$ such that $\sup P_w(\rho(\hat{f},U)) < \sup P_w(\rho(\hat{f}))$. We will show that this is not possible. Observe that there must be an extremal point $v\in \rho(\hat{f})$ such that $P_w(v) = \sup P_w(\rho(\hat{f}))$. Since $v$ is extremal, as we mentioned there exists $z\in \R^2$ such that $(\hat{f}^n(z)-z)/n\to v$ as $n\to \infty$. 

 Since $U$ is essential, $U'=\bigcup_{n\in \Z} f^n(U)$ is open, invariant, fully essential and connected (by Corollary \ref{coro:fully}). This implies that $\pi(z)$ is contained in some closed topological disk $D$ such that $\bd D\subset U'$. Since $\bd D$ is compact, there is $N\in \N$ such that $\bd D\subset \bigcup_{i=-N}^N f^i(U)$. Let $\hat{D}$ be the connected component of $\pi^{-1}(D)$ that contains $z$. Since $P_w((\hat{f}^n(z)-z)/n)\to P_w(v)$ as $n\to \infty$, if $z_n$ is chosen as a point of $\bd \hat{D}$ such that $P_w(\hat{f}^n(z_n)-z)$ is maximal then, as $\abs{P_w(z_n-z)}\leq \diam(\hat{D})$, 
 $$P_w(\hat{f}^n(z_n)-z_n)/n \geq P_w(\hat{f}^n(z)-z)/n - P_w(z_n-z)/n \xrightarrow{n\to \infty} P_w(v).$$
Let $K$ be such that $\norm{\smash{\hat{f}(y)-y}}\leq K$ for all $y\in \R^2$ (such $K$ exists because $\hat{f}-\id$ is $\Z^2$-periodic). Note that $\norm{\smash{\hat{f}^n(y)-y}}\leq nK$. Thus we may choose a subsequence $(n_i)_{i\in \N}$ such that $n_i\to \infty$ and $(\hat{f}^{n_i}(z_{n_i})-z_{n_i})/n_i$ converges to some limit $v'$ with $\norm{v'}\leq K$, and from our previous observations $P_w(v')\geq P_w(v)$. But also $P_w(v')\leq P_w(v)$, since we chose $v$ such that $P_w(v)=\sup P_w(\rho(\hat{f}))$. Therefore $P_w(v')=P_w(v)$.

 Observe that since $\pi(z_{n_i})\in \bd D$, we know that there is $k_i\in \Z$ with $-N\leq k_i\leq N$ such that $f^{k_i}(\pi(z_{n_i}))\in U$, so that if we let $x_i= \hat{f}^{k_{n_i}}(z_{n_i})$ then $x_i\in \pi^{-1}(U)$. Thus, letting $m_i=n_i-k_i$, we have 
$$\frac{\hat{f}^{m_i}(x_i) - x_i}{m_i} =\frac{n_i}{n_i-k_i}\cdot\frac{\hat{f}^{n_i}({z_{n_i}})-z_{n_i}}{n_i} - \frac{\hat{f}^{k_i}(z_{n_i})-z_{n_i}}{n_i}\xrightarrow{n\to \infty} v'$$
because $n_i\to \infty$, while $\abs{k_i}\leq N$ and $\norm{\smash{P_w(\hat{f}^{k_i}(z_{n_i})-z_{n_i})}}\leq k_iK\leq NK$ for all $n\in \N$. By definition, this means that $v'\in \rho(\hat{f},U)$. Since we already saw that $P_w(v')=P_w(v) = \sup P_w(\rho(\hat{f}))$, this contradicts our assumption that $\sup P_w(\rho(\hat{f},U)) < \sup P_w(\rho(\hat{f}))$. This completes the proof of the claim.
\end{proof}

To prove (1), observe that from the previous claim follows immediately that $\ess(f)\subset \mc{C}(f)$. Thus, we need to prove that $\mc{C}(f)\subset \ess(f)$, or which is the same that $\ine(f)\subset \mc{E}(f)$. Let $x\in \ine(f)$ then  by Theorem \ref{th:essine} the connected component $U$ of $\ine(f)$ that contains $x$ is a bounded periodic disk, so that $f^k(U)=U$ for some $k\in \N$. Thus if $\hat{U}$ is a connected component of $\pi^{-1}(U)$, there is $v\in \Z^2$ such that $\hat{f}^k(\hat{U})=\hat{U}+v$. This implies that $\norm{\smash{\hat{f}^{nk}(z)-z-nv}} \leq \diam(\hat{U})$ for all $z\in \hat{U}$, so we easily conclude that $\rho(\hat{f}^k, U)= \{v\}$, and then $\rho(\hat{f},U)=\{v/k\}\subset \Q^2$. This shows that $x\in \mc{E}(f)$, as we wanted. 

Therefore, we have proved (1) (since the claims about $\ess(f)$ hold by Theorem \ref{th:essine}). Further, the previous claim together with (1) implies (3). 

To prove (2), observe that the external sensitivity on initial conditions follows easily from (3), since it implies that if $U$ is a small ball around $x\in \mc{C}(f)$, and if $\hat{U}$ is a connected component of $\pi^{-1}(U)$, then $\diam(\hat{f}^n(U))\to \infty$ as $n\to \infty$. To prove the external transitivity, let $U_1,U_2$ be open sets in $\T^2$ intersecting $\mc{C}(f)=\ess(f)$. Then from Corollary \ref{coro:fully} $U_i'= \bigcup_{n\in \Z} f^n(U_i)$ is fully essential and invariant, for $i\in \{1,2\}$. But two fully essential sets must intersect, so there are $n_1, n_2\in \Z$ such that $f^{n_1}(U_1)\cap f^{n_2}(U_2)\neq \emptyset$, so that $\hat{f}^m(U_2)\cap U_1\neq \emptyset$, for $m=n_2-n_1\in \Z$. This completes the proof. \qed

\subsection{Proof of Theorem \ref{th:reali}} Let $f$ be homotopic to the identity, nonwandering and strictly toral, and let $\hat{f}$ be a lift of $f$ to $\R^2$. We want to prove 
\begin{enumerate}
\item[(1)] Any rational vector of $\rho(\hat{f})$ that is realized by some periodic point is also realized by a periodic point in $\ess(f)$).
\item[(2)] If $\mu$ is an ergodic Borel probability measure $\mu$ with associated rotation vector $\rho_\mu(\hat{f})\notin \Q^2$, then $\mu$ is supported on $\ess(f)$.
\end{enumerate}

We begin with (2): let $\mu$ be an $f$-ergodic Borel probability measure and $v=\rho_\mu(\hat{f})\notin \Q^2$. For $\mu$-almost every point $x\in \T^2$, we have that if $\hat{x}\in \pi^{-1}(x)$ then $(\hat{f}^n(\hat{x})-\hat{x})/n\to v$ as $n\to \infty$ (see \S\ref{sec:rotation}). Since $v\notin \Q^2$, this implies that $\pi(x)\in \ess(f)$, since otherwise by Theorem \ref{th:essine} it would belong to some periodic bounded disk $U$, and that would imply that $\rho(x)\in \Q^2$ (as in the proof of (1) in the previous section). Thus we conclude that $\mu$-almost every point is essential. Since $\ess(f)$ is closed and invariant, it follows that the support of $\mu$ is in $\ess(f)$, proving (2).

To prove (1) we will use a Lefschetz-Nielsen type index argument. Let $v=(p_1/q,p_2/q)\in \Q^2\cap \rho(\hat{f})$ with $p_1,p_2,q$ coprime. Let $g=f^q$ and $\hat{g}=\hat{f}^q-(p_1,p_2)$ (which is a lift of $g$). Recall that $z\in \T^2$ is a periodic point realizing the rotation vector $v$ (for $\hat{f}$) if for any $\hat{z}\in \pi^{-1}(z)$ one has $\hat{f}^q(\hat{z})-\hat{z} = v$. This is equivalent to saying that $\hat{g}(\hat{z}) = \hat{z}$. Therefore, to prove (1) we need to show that if $\fix(\hat{g})$ is nonempty, then $\pi(\fix(\hat{g}))$ intersects $\ess(f)$. By Corollary \ref{coro:fully} we have that $\ess(g) = \ess(f^k)=\ess(f)$. Thus we want to show that if $\fix(\hat{g})$ is nonempty then its projection to $\T^2$ contains a point of $\ess(g)$.

Suppose on the contrary that $K:=\pi(\fix(\hat{g}))\subset \ine(g)$. Since $K$ is compact, there are finitely many connected components $U_1,\dots U_k$ of $\ine(g)$ such that $K\subset U_1\cup \cdots \cup U_k$. Note that each $U_i$ is an open topological disk, and we may assume that each $U_i$ intersects $K\subset \fix(g)$, so $g(U_i)=U_i$ for each $i$. 

We claim that $\fix(g)\cap \ol{U}_i\subset K$ for each $i\in \{1,\dots, k\}$. Indeed, suppose $x\in \fix(g)\cap \ol{U}_i$, and choose a connected component $\hat{U}_i$ of $\pi^{-1}(U_i)$. If $(x_n)_{n\in \N}$ is a sequence in $U_i$ such that $x_n\to x$ as $n\to \infty$, and if $\hat{x}_n\in \pi^{-1}(x_n)\cap \hat{U}_i$, then the fact that $\diam(\hat{U}_i)<\infty$ implies that we may find a sequence $(n_i)_{n\in \N}$ with $n_i\to \infty$ such that $\hat{x}_{n_i}$ converges to some limit $\hat{x}\in \cl(\hat{U})$ as $i\to \infty$. Thus $\pi(\hat{x}) = \lim_{i\to \infty} \pi(\hat{x}_{n_i})= x$, and since $x\in \fix(g)$ it follows that $\hat{g}(\hat{x}) = \hat{x}+w$ for some $w\in \Z^2$. But $\cl(\hat{U})$ is bounded and $\hat{g}$-invariant, and since $\hat{g}^n(\hat{x}) = \hat{x}+nw$ we conclude that $w=0$. Hence $\hat{x}\in \fix(\hat{g})$, and $x\in K$, proving our claim.

In particular, since $K\subset U_1\cup\cdots \cup U_k$, we have that $\bd U_i$ contains no fixed points of $g$. Since $g$ is nonwandering, using a classic argument of Cartwright and Littlewood and the prime ends compactification of $U_i$, one may find a closed topological disk $D_i\subset U_i$ such that $\fix(g)\cap U_i\subset D_i$ and the fixed point index of $g$ in $D_i$ is $1$ (this is contained in Proposition 4.2 of \cite{koro}).
Thus we can cover $K = \fix(g)\cap (U_1 \cup \cdots \cup U_k)$ with finitely many disjoint disks $D_1,\dots D_k$ such that the fixed point index of $g$ on each $D_i$ is $1$. Note that $K$ is a Nielsen class of fixed points (that is, it consists of all points which are lifted to fixed points of a same lift $\hat{g}$ of $g$). We have just showed that the fixed point index of the Nielsen class $K$ is exactly $k\geq 1$ (one for each disk $D_i$, and there is at least one such disk). On the other hand, it is known (see, for instance, \cite{brown-nielsen}) that the fixed point index of a Nielsen class is invariant by homotopy, and since $f$ is homotopic to a map with no fixed points, the index should be $0$. Thus we arrived to a contradiction, completing the proof of (1).
\qed

\subsection{Proof of Corollary \ref{coro:trans}}
Let $f\colon \T^2\to \T^2$ be a strictly toral nonwandering homeomorphism. Suppose first that $f$ is not transitive. 
Note that the proof of external transitivity of $\ess(f)$ given in the proof of Theorem \ref{th:chaotic} works in the general case where $f$ is strictly toral (without assuming anything about the rotation set). Hence $\ess(f)$ is externally transitive. If $\ess(f)=\T^2$, this would imply that $f$ is transitive, contradicting our hypothesis. Thus $\ine(f)$ is nonempty, and the existence of a bounded periodic disk follows from Theorem \ref{th:essine}.

Now suppose that $f$ is transitive and assume for a contradiction that there is a periodic bounded disk of period $k$. Let $U_1,\dots, U_k$ be the components of the orbit of the disk, so that $f^k(U_i)=U_i$ for each $i$, and let $M$ be such that $\max_i \diamup(U_i)\leq M$. For each $i\in \{1,\dots, k\}$, let $\hat{U}_i$ be a connected component of $\pi^{-1}(U_i)$, and let $\hat{g}$ be a lift of $f^k$ such that $\hat{g}(\hat{U}_1)=\hat{U}_1$ (and therefore $\hat{g}(\hat{U}_i)=\hat{U}_i$ for each $i$). Note that $\cup_{i=1}^k\cl(U_i) = \T^2$. Thus, given $z\in \ol{U}_i$ we may choose $i$ such that $z\in \cl(U_i)$, and the fact that $U_i$ is bounded implies easily that some $\hat{z}\in \pi^{-1}(z)$ belongs to $\cl(\hat{U}_i)$. Hence, the $\hat{g}$-orbit of $\hat{z}$ has diameter bounded by $M$. This also holds for $\hat{z}+v$ for any $v\in \Z^2$, and since $z\in \T^2$ was arbitrary we conclude that all $\hat{g}$-orbits have diameter bounded by $M$. Since $\hat{g}$ lifts $f^k$, this implies that $f^k$ is annular, contradicting the fact that $f$ is strictly toral.\qed

\section{Brouwer theory and gradient-like foliations}\label{sec:brouwer}

Let $S$ be an orientable surface (not necessarily compact), and let $\mc{I}=(f_t)_{t\in [0,1]}$ be an isotopy from $f_0=\id_S$ to some homeomorphism $f_1=f$. If $\pi\colon \hat{S}\to \hat{S}$ is the universal covering of $S$, there is a natural choice of a lift $\hat{f}\colon \hat{S}\to \hat{S}$ of $f$: Letting $\hat{\mc{I}}=(\hat{f}_t)_{t\in [0,1]}$ be the lift of the isotopy $\mc{I}$ such that $\hat{f}_0=\id_{\hat{S}}$, one defines $\hat{f}=\hat{f}_1$. The lift $\hat{f}$ has the particularity that it commutes with every Deck transformation of the covering. 

A fixed point $p$ of $f$ is said to be \emph{contractible} with respect to the lift $\hat{f}$ if the loop $(f_t(p))_{t\in [0,1]}$ is homotopically trivial in $S$. This definition does not depend on the isotopy, but only on the lift $\hat{f}$. In fact, it is easy to see that the set of contractible fixed points of $f$ with respect to $\mc{I}$ coincides with $\pi(\fix(\hat{f}))$.

Given an oriented topological foliation $\mc{F}$ of $S$, one says that the isotopy $\mc{I}$ is transverse to $\mc{F}$ if for each $x\in S$, the arc $(f_t(x))_{x\in [0,1]}$ is homotopic, with fixed endpoints, to an arc that is positively transverse to $\mc{F}$ in the usual sense. In this case, it is also said that $\mc{F}$ is dynamically transverse to $\mc{I}$.

The following is one statement of the equivariant version of Brouwer's Plane Translation Theorem:
\begin{theorem}[Le Calvez \cite{lecalvez-equivariant}]\label{th:lecalvez}
If there are no contractible fixed points, then there is a foliation without singularities $\mc{F}$ which is dynamically transverse to $\mc{I}$.
\end{theorem}

Since the set of contractible fixed points is usually nonempty, one needs some additional modifications before using the previous theorem. This is done using a recent result of O. Jaulent.

\begin{theorem}[Jaulent, \cite{jaulent}] \label{th:jaulent} Given an isotopy $\mc{I}=(f_t)_{t\in [0,1]}$ from the identity to a homeomorphism $f\colon S\to S$, there exists a closed set $X\subset \fix(f)$ and an isotopy $\mc{I}' = (f_t')_{t\in [0,1]}$ from $\id_{S\sm X}$ to $f|_{S\sm X}\colon S\sm X\to S\sm X$ such that
\begin{enumerate}
\item[(1)] for each $z\in S\sm X$, the arc $(f_t'(z))_{t\in [0,1]}$ is homotopic with fixed endpoints (in $S$) to $(f_t(z))_{t\in [0,1]}$;
\item[(2)] there are no contractible fixed points for $f|_{S\sm X}$ with respect to $\mc{I}'$.
\end{enumerate}
\end{theorem}

\begin{remark}\label{rem:jaulent-1} Due to the latter property, Theorem \ref{th:lecalvez} implies that there is a foliation $\mc{F}_X$ on $S\sm X$ that is dynamically transverse to $\mc{I}'$. 
\end{remark}

\begin{remark}\label{rem:jaulent-2} If $X$ is totally disconnected, one can extend the isotopy $\mc{I}'$ to an isotopy on $S$ that fixes every element of $X$; that is, $f_t'(x)=x$ for each $x\in X$ and $t\in [0,1]$. Similarly, the foliation $\mc{F}_X$ can be extended to an oriented foliation with singularities $\mc{F}$ of $S$, where the set of singularities $\sing(\mc{F})$ coincides with $X$. Moreover, after these extensions, if we consider the respective lifts $\hat{\mc{I}}=(\hat{f}_t)_{t\in [0,1]}$ and $\hat{\mc{I}}'=(\hat{f}_t')_{t\in [0,1]}$ of $\mc{I}$ and $\mc{I}'$ such that $\hat{f}_0=\hat{f}_0'=\id_{\hat{S}}$, then $\hat{f}_1'=\hat{f}_1$. This follows from the fact that if $z\in S\sm X$, then $(f_t'(z))_{t\in [0,1]}$ is homotopic with fixed endpoints in $S$ to $(f_t(z))_{t\in [0,1]}$, so that the lifts of these paths with a common base point $\hat{z}$ must have the same endpoint as well. 
\end{remark}

\begin{remark}\label{rem:jaulent-3} In the previous remark, if $\hat{\mc{F}}$ is the lift of the extended foliation $\mc{F}$ (with singularities in $\hat{X}=\pi^{-1}(X)$), then $\hat{\mc{F}}|_{\hat{S}\sm \pi^{-1}(X)}$ is dynamically transverse to $\hat{\mc{I}}'$; \ie for any $\hat{z}\in \hat{S}\sm \hat{X}$ the path $(\hat{f}'_t(\hat{z}))_{t\in [0,1]}$ is homotopic with fixed endpoints in $\hat{S}\sm \hat{X}$ to an arc $\hat{\gamma}$ positively transverse to $\hat{\mc{F}}$. In fact we know that, if $z=\pi(\hat{z})$, then $(f'_t(z))_{t\in [0,1]}$ is homotopic with fixed endpoints in $S\sm X$ to an arc $\gamma$ positively transverse to $\mc{F}$. The homotopy between $(f'_t(z))_{t\in [0,1]}$ and $\gamma$ can be lifted to a homotopy (with fixed endpoints, in $\hat{S}\sm \hat{X}$) between $(\hat{f}'_t(z))_{t\in [0,1]}$ and the lift $\hat{\gamma}$ of $\gamma$ with base point $\hat{z}$. One easily verifies that $\hat{\gamma}$ is positively transverse to $\hat{\mc{F}}$.
\end{remark}

\subsection{Positively transverse arcs}

Let us state some general properties of dynamically transverse foliations that will be used in the next sections. This proposition is analogous to part of Proposition 8.2 of \cite{lecalvez-equivariant}, with small modifications.
\begin{proposition}\label{pro:pta} Suppose $S$ is an orientable surface, $\mc{I}=(f_t)_{t\in [0,1]}$ an isotopy from the identity to a homeomorphism $f=f_1$ without contractible fixed points, and $\mc{F}$ a dynamically transverse foliation as given by Theorem \ref{th:lecalvez}. The following properties hold:
\begin{enumerate}
\item [(1)] For any $n\in \N$ and $x\in S$, there is a positively transverse arc joining $x$ to $f^n(x)$.
\item [(2)] If $x$ and $y$ can be joined by a positively transverse arc, then there are neighborhoods $V$ of $x$ and $V'$ of $y$ such that every point of $V$ can be joined to every point of $V'$ by a positively transverse arc;
\item [(3)] If $x$ is nonwandering, then there is a neighborhood of $V$ of $x$ such that every point of $V$ can be joined to every other point of $V$ by a positively transverse arc.
\item [(4)] If $K\subset S$ is a connected set of nonwandering points, then any point of $K$ can be joined to each point of $K$ by a positively transverse arc.
\end{enumerate}
\end{proposition}

\begin{remark} Note that this proposition remains true in the context of Remarks \ref{rem:jaulent-2} and \ref{rem:jaulent-3} if one works in $S\sm X$ or $\hat{S}\sm \hat{X}$ with the corresponding foliations.
\end{remark}

\begin{proof} The first claim is a consequence of the transversality of the foliation: we know that any $z\in S$ can be joined to $f(z)$ by some positively transverse arc $\gamma_z$, and so $\gamma_x^n = \gamma_x*\gamma_{f(x)}*\cdots*\gamma_{f^{n-1}(x)}$ is a positively transverse arc joining $x$ to $f^n(x)$.

\begin{figure}[ht]
\centering
\includegraphics[width = \textwidth]{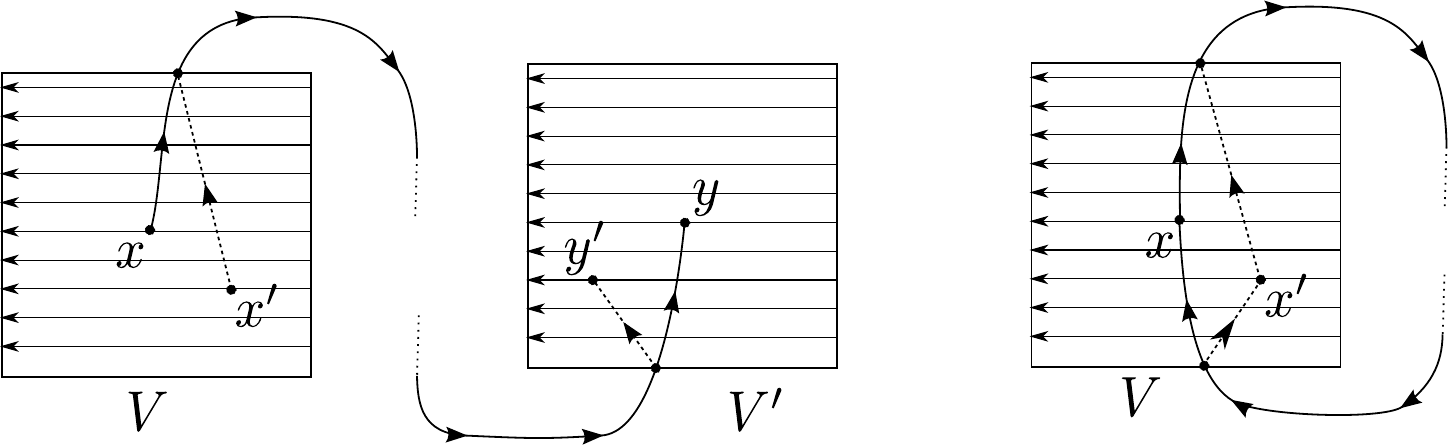}
\caption{Proofs of $(2)$ and $(3)$}
\label{fig:pta3}
\end{figure}

To prove (2) it suffices to consider flow boxes of $\mc{F}$ near $x$ and $y$ (see the left side of Figure \ref{fig:pta3}). Similarly, to prove (3) it suffices to show that if $x$ is nonwandering then there is a positively transverse arc joining $x$ to itself (see the right side of Figure \ref{fig:pta3}). Observe that due to (1) and (2), we can find neighborhoods $V$ of ${f}^{-1}(x)$ and $V'$ of $x$ such that every point of $V$ can be joined to every point of $V'$ by a positively transverse arc, and reducing $V'$ if necessary we may also find a neighborhood $V''$ of $f(x)$ such that every point of $V'$ can be joined to every point of $V''$ by a positively transverse arc. Since $x$ is nonwandering, we can find $y\in V'$ so close to $x$ that $f(y)\in V''$ and such that $f^n(y)\in V'\cap f(V)$ for some $n>0$ which we may assume large. Thus we can find a positively transverse arc from $x\in V'$ to $f(y)\in V''$, another from $f(y)$ to $f^{n-1}(y)$ (namely, $\gamma^{n-2}_{f(y)}$), and a third one from $f^{n-1}(y)\in V$ to $x\in V'$. Concatenation of these arcs gives a positively transverse arc from $x$ to itself, as we wanted.

Finally, to prove $(4)$, fix $x\in K$ and let $K'$ be the set of all points of $K$ which are endpoints of positively transverse arcs starting at $x$. From (2) follows that $K'$ is open in $K$, and from (3) we know that $x\in K'$. Moreover, $(3)$ also implies $K\sm K'$ is open in $K$. Thus $K'$ is both open and closed in $K$, and the connectedness implies $K=K'$.
\end{proof}

\subsection{Gradient-like foliations}\label{sec:gradient}

Let $\mc{F}$ be an oriented foliation with singularities of $\T^2$ such that $\sing(\mc{F})$ is totally disconnected. A leaf $\Gamma$ of $\mc{F}$ is a \emph{connection} if both its $\omega$-limit and its $\alpha$-limit are one-element subsets of $\sing(\mc{F})$. By a \emph{generalized cycle of connections} of $\mc{F}$ we mean a loop $\gamma$ such that $[\gamma]\sm \sing(\mc{F})$ is a disjoint (not necessarily finite) union of regular leaves of $\mc{F}$, with their orientation matching the orientation of $\gamma$.

Using the terminology of Le Calvez \cite{lecalvez-equivariant}, we say that a loop $\Sigma$ in $\T^2$ is a \emph{fundamental loop} for $\mc{F}$ if $\Sigma$ can be written as a concatenation of finitely many loops $\alpha_1,\dots,\alpha_n$ with a common base point $z_0$ such that, denoting by $\alpha_i^*\in H^1(\T^2,\Z)\simeq \Z^2$ the homology class of $\alpha_i$, 
\begin{itemize}
\item each loop $\alpha_i$ is positively transverse to $\mc{F}$,
\item $\sum_{i=1}^n \alpha_i^*=0$, and
\item for each $\kappa \in H^1(\T^2,\Z)$ there are positive integers $k_1,\dots,k_n$ such that $\sum_{i=1}^n k_i \alpha_i^* = \kappa$.
\end{itemize}

It is easy to see that if $\Sigma$ is a fundamental loop, then $\T^2\sm [\Sigma]$ is a disjoint union of open topological disks.

If there exists a fundamental loop $\Sigma$ for $\mc{F}$, we say that $\mc{F}$ is \emph{gradient-like}. The key properties about gradient-like foliations that we will use are contained in the following

\begin{lemma} \label{lem:gradient} If $\mc{F}$ is a gradient-like foliation, then 
\begin{itemize} 
\item[(1)] every regular leaf of $\mc{F}$ is a connection,
\item[(2)] there are no generalized cycles, and
\item[(3)] there is a constant $M$ such that $\diamup(\Gamma)<M$ for each regular leaf $\Gamma$.
\end{itemize}
\end{lemma}

\begin{proof}
We outline the proof, since the main ideas are contained in \S 10 of \cite{lecalvez-equivariant}. The difference here is that we do not have finitely many singularities. Let $\Sigma$ be a fundamental loop with base point $x_0$ as defined above, so $\T^2\sm [\Sigma]$ is a disjoint union of simply connected open sets.  After a perturbation of the arcs $\alpha_i$ to put them in general position, one may assume that $\T^2\sm [\Sigma]$ has finitely many connected components, so it is a disjoint union of finitely many topological disks $\{D_i\}_{1\leq i\leq k}$.

A function $\Lambda$ is then defined on $\T^2 \sm [\Sigma]$ by fixing a point $z_0\in \T^2\sm [\Sigma]$ and letting $\Lambda(z)$ be the algebraic intersection number $\sigma\wedge \Sigma$ of any arc $\sigma$ joining $z_0$ to $z$ with $\Sigma$. This is independent of the choice of $\sigma$, because $\Sigma^*=0$. The function $\Lambda$ is constant on each disk $D_i$, and it has the property that if $\sigma$ is an arc joining $z\in \T^2\sm [\Sigma]$ to $z'\in \T^2\sm [\Sigma]$, then $\sigma\wedge \Sigma =\Lambda(z')-\Lambda(z)$.  Note that $\Lambda$ attains at most $k$ different values (one for each disk $D_i$). The fact that $\Sigma$ is positively transverse to $\mc{F}$ implies that if $\Gamma\colon \R\to \T^2$ is any leaf of $\mc{F}$, then the map $t\mapsto \Lambda(\Gamma(t))$ (defined for all $t$ such that $\Gamma(t)\notin [\Sigma]$) is non-increasing, and it decreases after each $t$ such that $\Gamma(t)\in [\Sigma]$. 

The proof of part (i) in Proposition 10.4 of \cite{lecalvez-equivariant} shows that $\mc{F}$ has no closed leaves: 
Suppose $\Gamma$ is a closed leaf of $\mc{F}$, and let $z$ be a point of $\Gamma$. Since there are no wandering points, by Proposition \ref{pro:pta} there is a positively transverse loop $\gamma$ based in $z$. This implies that $\Gamma\wedge \gamma<0$. On the other hand, there exist positive integers $a_1,\dots, a_n$ such that $-\gamma^*=a_1\alpha_1^*+\dots+a_n\alpha_n^*$, so that letting $\gamma'=\alpha_1^{a_1}*\cdots*\alpha_n^{a_n}$ one has $\Gamma\wedge\gamma'=\Gamma\wedge (-\gamma)=-\Gamma\wedge\gamma$. Thus 
$$0>\Gamma\wedge \gamma = -\Gamma\wedge \gamma' = -(a_1\Gamma\wedge \alpha_1+\cdots+a_n\Gamma\wedge\alpha_n) \geq 0,$$
where the latter inequality holds because $\alpha_i$ is a positively transverse arc and $a_i$ is a positive integer, for each $i$. This contradiction proves that $\mc{F}$ has no closed leaves.

To show that there is no cycle of connections, first observe that by definition if there is a cycle of connections, it contains a \emph{simple} cycle of connections; that is, a simple loop $\Gamma$ such that $[\Gamma]\sm X$ consists of leaves of $\mc{F}$ with their orientation matching the orientation of $\Gamma$. But then, choosing $z\in [\Gamma]\sm X$ we can repeat the previous argument by finding a positively transverse loop $\gamma$ based in $z$, and obtaining the same contradiction as before. This proves (2).

Recall that if $\Gamma\colon \R\to \T^2$ is a leaf of $\mc{F}$, the map $t\mapsto \Lambda(\Gamma(t))$ defined on $\R\sm \Gamma^{-1}([\Sigma])$ is non-increasing, and it decreases after each $t$ such that $\Gamma(t)\in [\Sigma]$. Since $\Lambda$ attains at most $k$ different values (one for each disk $D_i$), it follows that $\Gamma$ intersects $\Sigma$ at most at $k$ points.

Let $\hat{\mc{F}}$ be the lift of $\mc{F}$ to $\R^2$, and for each $i$, let $\hat{D}_i$ be a connected component of $\pi^{-1}(D_i)$, and let $\mc{B} = \{\hat{D}_i+v: v\in \Z^2, 1\leq i\leq k\}$. Then from the previous paragraph follows that any regular leaf $\hat{\Gamma}$ of $\mc{\hat{F}}$ intersects at most $k+1$ elements of $\mc{B}$. Let $d=\max\{\diam{\hat{D}_i}:1\leq i\leq k\}$. Note that $\diam(D)\leq d$ for each $D\in \mc{B}$. We conclude from these facts that $\diam(\hat{\Gamma})\leq M \doteq (k+1)d$, proving part (3).

Finally, part (1) follows from the following version of the Poincar\'e-Bendixson Theorem, which is a particular case of a theorem of Solntzev \cite{solntzev} (see also \cite[\S 1.78]{stepanov}) and can be stated in terms of continuous flows due to a theorem of  Gutierrez \cite{gutierrez}.
\begin{theorem} Let $\phi=\{\phi_t\}_{t\in \R}$ be a continuous flow on $\R^2$ with a totally disconnected set of singularities. If the forward orbit of a point $\{\phi_t(z)\}_{t\geq 0}$ is bounded, then its $\omega$-limit $\omega_{\phi}(z)$ is one of the following:
\begin{itemize} 
\item A singularity;
\item a closed orbit;
\item a generalized cycle of connections.
\end{itemize}
\end{theorem}
 
Since, being an oriented foliation with singularities, $\hat{\mc{F}}$ can be embeded in a flow (see \cite{whitney, whitney2}), we may apply the above theorem to $\hat{\mc{F}}$. Since $\mc{F}$ has no generalized cycle of connections or closed leaves, neither does $\hat{\mc{F}}$, and we conclude that the $\omega$-limit of every bounded leaf of $\hat{\mc{F}}$ is a singularity (and similarly for the $\alpha$-limit). Since we already showed that every leaf is bounded, this proves (1), completing the proof of Lemma \ref{lem:gradient}.
\end{proof}

\subsection{Existence of gradient-like Brouwer foliations}
\label{sec:gradient-brouwer}

Throughout this section we assume that $f$ is a homeomorphism of $\T^2$ isotopic to the identity and $\hat{f}$ is a lift of $f$ to $\R^2$ such that $\fix(\hat{f})$ is totally disconnected, hence so is $\pi(\fix(\hat{f}))$.

We observe that there exists an isotopy from the identity to $f$ that lifts to an isotopy from $\id_{\R^2}$ to $\hat{f}$: indeed, it suffices to choose any isotopy $(f_t)_{t\in [0,1]}$ from the identity to $f$ and its lift $(\hat{f}_t)_{t\in [0,1]}$ such that $\hat{f}_0=\id_{\R^2}$. Noting that there is some $v\in \Z^2$ such that $\hat{f}-\hat{f}_1=v$, the isotopy from $\id_{\T^2}$ to $f$ lifted by $(\hat{f}_1+tv)_{t\in [0,1]}$ has the required property. 

The next proposition is a direct consequence of Theorem \ref{th:jaulent} and the remarks that follow it. 

\begin{proposition} \label{pro:brouwer} There exists an oriented foliation with singularities $\mc{F}$ of $\T^2$ and an isotopy $\mc{I}=(f_t)_{t\in [0,1]}$ from the identity to $f$ such that
\begin{itemize}
\item $\sing(\mc{F})\subset \pi(\fix(\hat{f}))$,
\item $\mc{I}$ lifts to an isotopy from $\id_{\R^2}$ to $\hat{f}$, 
\item $\mc{F}$ is dynamically transverse to $\mc{I}$, and
\item $\mc{I}$ fixes the singularities of $\mc{F}$.
\end{itemize}
\end{proposition}

Let $\mc{F}$ be the foliation from Proposition \ref{pro:brouwer}. Recall that for a loop $\gamma$ in $\T^2$, $\gamma^*$ denotes its homology class in $H_1(\T^2, \Z)\simeq \Z^2$. Fix $z\in \T^2\sm X$, and consider the set $\mc{C}(z)$ of all homology classes $\kappa\in H^1(\T^2,\Z)$ such that there is a positively transverse loop $\gamma$ with $\gamma^* = \kappa$. Identifying $H^1(\T^2,\Z)$ with $\Z^2$ naturally and choosing $\hat{z}\in \pi^{-1}(z)$, we see that $\mc{C}(z)$ coincides with the set of all $v\in \Z^2$ such that there is an arc in $\R^2$ positively transverse to the lifted foliation $\hat{\mc{F}}$ joining $\hat{z}$ to $\hat{z}+v$. Note that $\mc{C}(z)$ is closed under addition: if  $v,w\in \mc{C}(z)$ then $v+w\in \mc{C}(z)$. 

The next proposition is contained in Lemma 10.3 and the first paragraph after its proof in \cite{lecalvez-equivariant}. The proof given there works without modifications in our context.

\begin{proposition}\label{pro:zerohull} If $f$ is nonwandering and the convex hull of $\mc{C}(z)$ is $\R^2$ for some $z\in \T^2$, then there is a fundamental loop.
\end{proposition}

\begin{remark} \label{rem:zerohull} Note that $\R^2$ is the convex hull of $\mc{C}(z)$ if $0$ is in the interior of the convex hull of $\mc{C}(z)$, due to the fact that if $v\in \mc{C}(z)$ then $nv\in \mc{C}(z)$ for any $n\in \N$. Moreover, to show that the convex hull of $\mc{C}(z)$ contains $0$ in its interior it suffices to find $n$ positively transverse loops $\gamma_1,\dots,\gamma_n$ (not necessarily with base point $z$) such that $0$ is in the interior of the convex hull of $\{\gamma_1^*, \dots, \gamma_n^*\}$. In fact, note that if $z\notin X$, then using the fact that $f$ is nonwandering and $\T^2\sm X$ is connected, Proposition \ref{pro:pta} implies that for each $i$ we may find positively transverse arcs $\sigma_i$ from $z$ to $\gamma_i(0)$ and $\sigma'$ from $\gamma_i(0)$ to $z$. For $m\in \N$, define $\eta_{i,m} = \sigma_i*\gamma_i^m*\sigma_i'$ for $1\leq i\leq n$. Then $\eta_{i,m}$ is a positively transverse arc with base point $z$, and $\eta_{i,m}^* = (\sigma_i*\sigma_i')^* + m\gamma_i^* = w_i+m\gamma_i^*$ where $w_i$ is independent of $m$. Since $0$ is in the interior of the convex hull of $\{\gamma_i^* : 1\leq i\leq n\}$, choosing $m$ large enough it follows easily that $0$ is in the interior of the convex hull of $\{\eta_{i,m}: 1\leq i\leq n\}\subset C(z)$, as claimed.
\end{remark}

\section{Linking number of simply connected open sets}\label{sec:linking}

In this section we assume that $\hat{\mc{I}}=(\hat{f}_t)_{t\in [0,1]}$ is an isotopy from $\id_{\R^2}$ to a homeomorphism $\hat{f}\colon \R^2\to \R^2$, and $\hat{X}$ is a closed set of fixed points of the isotopy $\hat{\mc{I}}$, \ie $\hat{f}_t(p)=p$ for all $t\in [0,1]$ and $p\in \hat{X}$.

\subsection{Winding number}
Given $z\in \R^2$ and an arc $\gamma\colon [0,1]\to \R^2$ such that $z\notin [\gamma]$, we define a partial index as follows: consider the map
$$\xi\colon [0,1]\to \SS^1, \quad \xi(t) = \frac{\gamma(t)-z}{\norm{\gamma(t)-z}}$$
and let $\til{\xi}\colon [0,1]\to \R$ be a lift to the universal covering, so that $e^{2\pi\til{\xi}(t)} = \xi(t)$. 
Then we define 
$$I(\gamma,z) = \til{\xi}(1)-\til{\xi}(0).$$
This number does not depend on the choice of the lift $\til{\xi}$ or the parametrization of $\gamma$ (preserving orientation). If $\gamma$ is a loop, then $I(\gamma,z)$ is an integer and coincides with the winding number of $\gamma$ around $z$.
If $\gamma$ and $\gamma'$ are arcs with $\gamma(1)=\gamma'(0)$ and $z\notin [\gamma]\cup[\gamma']$, then
$$I(\gamma*\gamma',z) = I(\gamma,z)+I(\gamma',z).$$
Additionally, $I(\gamma,z)$ is invariant by homotopies in $\R^2\sm \{z\}$ fixing the endpoints of $\gamma$. A simple consequence of this fact is that if $I(\gamma,z)\neq 0$ and $\gamma$ is closed, then $z$ must be in a bounded connected component of $\R^2\sm [\gamma]$.

\subsection{Linking number of periodic points} 

\begin{notation} Given $z\in \R^2$, we denote by $\hat{\gamma}_z$ the arc $(\hat{f}_t(z))_{t\in [0,1]}$, and for $n\in \N$ we define $$\hat{\gamma}_z^n = \hat{\gamma}_z*\hat{\gamma}_{f(z)}*\cdots*\hat{\gamma}_{f^{n-1}(z)}.$$ 
\end{notation}

If $p\in \hat{X}$ (so $p$ is fixed by $\hat{\mc{I}}$) and $q$ is a periodic point of $\hat{f}$, then we define the linking number $I_{\hat{\mc{I}}}(q,p)\in \Z$ as follows. Let $k$ be the smallest positive integer such that $\hat{f}^k(q)=q$. Observing that $\hat{\gamma}_q^k$ is a loop, we let
$$I_{\hat{\mc{I}}}(q,p) = I(\hat{\gamma}_q^k, p).$$

We will extend this definition, considering a periodic (possibly unbounded) simply connected set instead of the periodic point $q$. 

\subsection{Linking number of open periodic simply connected sets}

\begin{definition}[and Claim]\label{def:index-U} Suppose $U\subset \R^2$ is a simply connected $\hat{f}$-periodic open set and $p\in \hat{X}\sm U$ is given. Let $k$ be the smallest positive integer such that $\hat{f}^k(U)=U$. Fix $z\in U$, and let $\sigma_z$ be an arc contained in $U$ and joining $\hat{f}^k(z)$ to $z$. The \emph{linking number} of $U$ and $p$ is defined as $I_{\hat{\mc{I}}}(U,p) = I(\hat{\gamma}^k_z*\sigma_z, p)$. This number does not depend on the choice of $z$ or the arc $\sigma_z$ in $U$.
\end{definition}

\begin{proof}[Proof of the claim]
First observe that $I_{\mc{\hat{I}}}(U,p)$ does not depend on the choice of $\sigma_z$ because if $\sigma'_z$ is any other arc in $U$ joining $\hat{f}^k(z)$ to $z$, then $I(\sigma_z*(-\sigma'_z),p)=0$ because $p\notin U$, and $U$ is simply connected. Thus $I(\hat{\gamma}^k_z*\sigma_z, p) =I(\hat{\gamma}^k_z, p)+ I(\sigma_z,p) = I(\hat{\gamma}^k_z, p)+I(\sigma_z',p) = I(\hat{\gamma}^k_z*\sigma_z', p)$ as required.

Now let $z'$ be another point in $U$, and fix an arc $\eta$ in $U$ joining $z$ to $z'$. 
We use the notation $\eta^s(t) = \eta|_{[0,s]}(st)$ (so $\eta^s$ is the sub-arc of $\eta$ from $\eta(0)$ to $\eta(s)$).  Letting $\sigma_{z'}=(-\hat{f}^k\circ\eta)*\sigma_z*\eta$, which is an arc in $U$ joining $\hat{f}^k(z')$ to $z'$, we have a homotopy 
$$\left(\hat{\gamma}^k_{\eta(s)}*(-\hat{f}^k\circ\eta^s)*\sigma_z*\eta^s\right)_{s\in [0,1]}$$
from $\hat{\gamma}^k_z*\sigma_z$ to $\hat{\gamma}^k_{z'}*\sigma_{z'}$ in $\R^2\sm \{p\}$, and therefore
$$I(\hat{\gamma}^k_z*\sigma_z,p)= I(\hat{\gamma}^k_{z'}*\sigma_{z'},p),$$
proving the independence on the choice of $z$.
\end{proof}

As a consequence of the independence on the choice of $z$ or $\sigma_z$ in the previous definition, we obtain the following

\begin{proposition}\label{pro:link-periodic} Let $U$ be the set from Definition \ref{def:index-U}, and suppose that there is $q\in U$ such that $\hat{f}^k(q)=q$ (where $k$ is the smallest positive integer such that $\hat{f}^k(U)=U$). Then $I_{\hat{\mc{I}}}(U,p) = I_{\hat{\mc{I}}}(q,p) = I(\hat{\gamma}^k_q,p)$ for any $p\in \hat{X}\sm \hat{U}$.
\end{proposition}

The proof is immediate by using $z=q$ and the constant arc $\sigma_z(t)=z$ in Definition \ref{def:index-U}.

\subsection{A linking lemma}
The following lemma is key in the proof of Theorem \ref{th:bdfix}; it is particularly useful when working with a gradient-like Brouwer foliation. Note that we are not assuming in this section that $\hat{f}$ is a lift of a torus homeomorphism.

\begin{lemma}\label{lem:inter-link} Let $U\subset \R^2$ be an open simply connected $\hat{f}$-periodic set, and assume that there are no wandering points of $\hat{f}$ in $U$. Let $\hat{\mc{F}}$ be an an oriented foliation with singularities of $\R^2$ such that $\hat{X}=\sing(\hat{\mc{F}})$ and for each $z\in \R^2\sm \hat{X}$, the arc $\hat{\gamma}_z$ is homotopic with fixed endpoints in $\R^2\sm \hat{X}$ to an arc positively transverse to the foliation. 

Suppose $\Gamma$ is a leaf of $\hat{\mc{F}}$ joining $p\in \hat{X}\sm U$ to $q\in \hat{X}\sm U$ and intersecting $\ol{U}$. Then either $I_{\mc{\hat{I}}}(U,p)\neq 0$ or $I_{\mc{\hat{I}}}(U,q)\neq 0$.
\end{lemma}
\begin{proof} 
Let $A$ be the annulus obtained by removing the points $p,q$ from the one-point compactification $\R^2\cup\{\infty\}$ of 
$\R^2$; that is, $A = \R^2\cup\{\infty\}\sm \{p,q\}$, and let $\tau\colon\til{A}\to A$ be the universal covering. Note that the isotopy $\hat{\mc{I}}|_{\R^2\sm\{p,q\}}$ extends to $A$ by fixing the point at $\infty$, and this extension lifts to an isotopy $\til{\mc{I}}=(\til{f}_t)_{t\in [0,1]}$ from $\id_{\til{A}}$ to some map $\til{f}=\til{f}_1$, which commutes with the group of covering transformations $\deck(\tau)$. The foliation $\hat{\mc{F}}|_{\R^2\sm \{p,q\}}$ also extends to $A$ by adding a singularity at $\infty$, and this extension lifts to a foliation $\til{\mc{F}}$ of $\til{A}$ with singularities in $\til{X} = \tau^{-1}(X\cup\{\infty\}\sm \{p,q\})$. 

Because $\mc{F}$ is dynamically transverse to $\mc{I}$, one easily sees that $\til{\mc{F}}$ is also dynamically transverse to $\til{\mc{I}}$; \ie, if $z\in \til{A}$ is not fixed by $\til{\mc{I}}$, then the arc $(\til{f}_t(z))_{t\in[0,1]}$ is homotopic with fixed endpoints in $A\sm \til{X}$, to an arc positively transverse to $\til{\mc{F}}$. 

Let $\til{U}$ be a connected component of $\tau^{-1}(U)$. Then $\til{U}$ is simply connected, and $\tau|_{\til{U}}$ is injective. Moreover, $\til{f}^k(\til{U})=T\til{U}$ for some covering transformation $T\in \deck(\tau)$, where $k$ is the least positive integer such that $f^k(U)=U$.

We will show that $T\neq \id$. Suppose for contradiction that $\til{f}^k(\til{U})=\til{U}$.  Let $z\in [\Gamma]\cap \cl(U)$, choose $\til{z}\in \tau^{-1}(z)$, and let $\til{\Gamma}$ be the leaf of $\til{\mc{F}}$ through $\til{z}$ (so that $\tau(\til{\Gamma})=\Gamma$). 
From the fact that the $\omega$-limit and $\alpha$-limit of $\Gamma$ are $q$ and $p$, respectively, it follows that $\til{\Gamma}$ is a proper embedding of $\R$ in $\til{A}\simeq \R^2$. Thus $\til{A}\sm [\til{\Gamma}]$ has exactly two connected components, and the fact that $\til{\mc{F}}$ is dynamically transverse implies that $\Gamma$ is a Brouwer line; \ie, $\til{f}(\til{\Gamma})$ and $\til{f}^{-1}(\til{\Gamma})$ belong to different connected components of $\til{A}\sm [\til{\Gamma}]$. This implies that one of the connected components $V$ of $\til{A}\sm [\til{\Gamma}]$ satisfies $\til{f}(\cl{V})\subset V$. It follows from this fact that every point of $[\til{\Gamma}]$ is wandering for $\til{f}$; in particular $\til{z}$ is wandering for $f$, so there is a neighborhood $W$ of $\til{z}$ such that $\til{f}^{n}(W)\cap W=\emptyset$ for all $n\in \N$. But $\tau(W)\cap U\neq \emptyset$, because $z= \tau(\til{z})\in \bd U$; thus we can find $T'\in \deck(\tau)$ such that $T'\til{U}\cap W\neq \emptyset$ (see Figure \ref{fig:prelimi3}). Since $\til{f}$ commutes with the Deck transformations, it follows that $\til{f}^k(T'\til{U})=T'\til{U}$, and since $\tau|_{T'\til{U}}$ is injective and $f^k$ has no wandering points in $U$, we conclude that $\til{f}^k$ has no wandering points in $T'\til{U}$, contradicting the fact that $\til{U}\cap W\neq \emptyset$. 

\begin{figure}[ht]
\includegraphics[height=5cm]{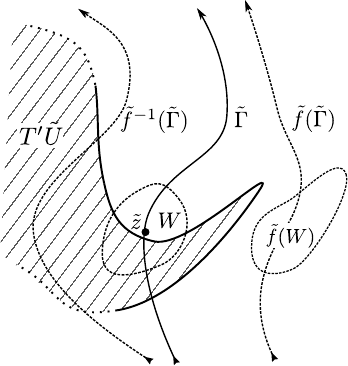}
\caption{Proof of Lemma \ref{lem:inter-link}}
\label{fig:prelimi3}
\end{figure}

Thus $T\neq \id$. Fix $\til{z}\in \til{U}$, and as in the previous section, let $\til{\gamma}_{\til{z}}$ denote the loop $(\til{f}(\til{z}))_{t\in [0,1]}$ and $\til{\gamma}_{\til{z}}^k = \til{\gamma}_{\til{z}}*\til{\gamma}_{\til{f}(\til{z})}\cdots *\til{\gamma}_{\til{f}^{k-1}(\til{z})}$. Choose any arc $\til{\sigma}_{\til{z}}$ in $T\til{U}$ joining $\til{f}^k(\til{z})$ to $T\til{z}$. Then letting $z=\tau(\til{z})$ and $\sigma_z = \tau\circ \til{\sigma}_z$, it follows that $\tau\circ(\til{\gamma}^k_{\til{z}}*\til{\sigma}_z) = \hat{\gamma}^k_z*\sigma_z$ is a homotopically nontrivial loop in $A$, since it lifts to a loop joining $\til{z}$ to $T\til{z}$. Of course it is still homotopically nontrivial in $A\sm \{\infty\} = \R^2\sm\{p,q\}$. This means that $I(\hat{\gamma}^k_z*\sigma_z, p)\neq 0$ or $I(\hat{\gamma}^k_z*\sigma_z,q)\neq 0$. Since $\sigma_z$ is an arc in $U$ joining $\hat{f}^k(z)$ to $z$, it follows from the definition that $I_{\mc{\hat{I}}}(U,p)\neq 0$ or $I_{\mc{\hat{I}}}(U,q)\neq 0$, as claimed.
\end{proof}

\begin{remark} Looking at the above proof in more detail, one may conclude the following more precise statement: There is $k>0$ such that $I(p,U) + I(q,U) = k$.  To see this, we may choose a simple loop $\alpha$ in $\R^2$ that bounds a disk containing $p$ but not $q$, with $\alpha$ oriented clockwise, as a generator of $\deck(A)$. That is, we may assume that $\deck(A) = \{T_0^k : k\in \Z\}$ where $T_0$ is a covering transformation of $\tau$ such that $T_0(\til{\alpha}(0))=\til{\alpha}(1)$, where $\til{\alpha}$ is any lift of $\alpha$ to $\til{A}$. Further, we may choose $\alpha$ such that it is positively transverse to $\Gamma$. In this setting, when we conclude that $T\neq \id$ in the proof above, the orientation of $\Gamma$ (from $p$ to $q$) implies that $T=T_0^k$ for some $k>0$. Therefore the loop $\hat{\gamma}^k*\sigma_z$ is homotopic to $\alpha^k$ in $A$, so that $I(\hat{\gamma}^k*\sigma_z,p)+I(\hat{\gamma}^k*\sigma_z,q) = I(\alpha^k, p)+I(\alpha^k,q)$. One can conclude easily from this fact that $I(p,U)+I(q,U) = k$. 
\end{remark}

\subsection{Application to gradient-like foliations}

Let us assume in this subsection the same hypotheses of \S\ref{sec:gradient-brouwer}, \ie $f\colon \T^2\to \T^2$ is a nonwandering homeomorphism homotopic to the identity with a totally disconnected set of fixed points and $\hat{f}$ is a lift of $f$. Let $\mc{F}$ and $\mc{I}$ be the oriented foliation with singularities and the isotopy given by Proposition \ref{pro:brouwer}, so that
\begin{itemize}
\item $\sing(\mc{F})\subset \pi(\fix(\hat{f}))$,
\item $\mc{I}$ lifts to an isotopy $\hat{\mc{I}} = (\hat{f}_t)_{t\in [0,1]}$ from $\id_{\R^2}$ to $\hat{f}$, 
\item $\mc{F}$ is dynamically transverse to $\mc{I}$, and
\item $\mc{I}$ fixes the singularities of $\mc{F}$ (and $\hat{\mc{I}}$ fixes the singularities of $\hat{\mc{F}}$)
\end{itemize}

We assume additionally that $\mc{F}$ is gradient-like. Denote $\hat{\mc{F}}$ the lift of $\mc{F}$ to $\R^2$, and let $\hat{X}=\sing(\mc{F})$.

\begin{proposition}\label{pro:large-X}
For each $k\in \N$, there is a constant $M_k$ such that if $U\subset \R^2$ is an open simply connected $\hat{f}^k$-invariant set without wandering points and $\diam(U)>M_k$, then $U\cap \hat{X}\neq \emptyset$.
\end{proposition}

\begin{proof} 
Since $\mc{F}$ is gradient-like, there is a constant $M'$ such that every regular leaf $\Gamma$ of $\hat{\mc{F}}$ connects two different elements of $\hat{X}$ and $\diam(\Gamma)<M'$. 
Let $$M'' = \sup \left\{\norm{\smash{\hat{f}_t(x)-\hat{f}_s(x)}} : s,t\in [0,1], x\in [0,1]^2\right\}.$$ 
From the fact that $\hat{f}$ commutes with integer translations, we have that $\diam([\hat{\gamma}_x])\leq M''$ for any $x\in \R^2$.

Let $k$ be the smallest positive integer such that $\hat{f}^k(U)=U$. 
Since $\hat{f}^k|_{U}$ is nonwandering (because $\hat{f}|_U$ is), Proposition \ref{pro:brouwer-trivial} implies that $\hat{f}^k$ has a fixed point $z$ in $U$. Define 
$$A=\left\{p\in \hat{X}\sm U: I_{\hat{\mc{I}}}(z,p) \neq 0\right\}.$$
Observe that $A$ coincides with the set of all $p\in \hat{X}\sm U$ such that $I(\hat{\gamma}^k_{z},p)\neq 0$, which is contained in the convex hull of $[\hat{\gamma}^k_{z}]$. Since $\diam([\hat{\gamma}^k_{z}])\leq \sum_{i=0}^{k-1} \diam([\hat{\gamma}_{f^i(z)}]) \leq kM''$, we conclude that  $A\subset B_{kM''}(z)$ (the ball of center $z$ and radius $kM''$).

On the other hand, by Proposition \ref{pro:link-periodic} it follows that
$$A = \left\{p\in \hat{X} : I_{\hat{\mc{I}}}(U,p) \neq 0\right\}.$$
\begin{figure}[ht]
\centering
\includegraphics[height=4cm]{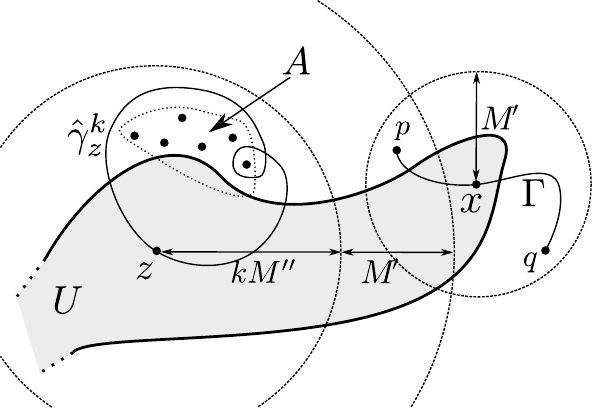}
\caption{Proof of Proposition \ref{pro:large-X}}
\label{fig:appli1}
\end{figure}
Suppose that $\diam(U)>M_k\doteq 2(kM''+M')$. We claim that $U$ intersects $\hat{X}$. Suppose on the contrary that $U\cap \hat{X}=\emptyset$. There is some point $x\in U\sm \ol{B}_{kM''+M'}(z)$, and by our assumption $x\notin \hat{X}$. See Figure \ref{fig:appli1}.
The leaf $\Gamma$ of $\hat{\mc{F}}$ through $x$ is such that $\diam(\Gamma)<M'$, and so its endpoints are two elements of $p,q$ of $\hat{X}\sm \ol{B}_{kM''}(z)$. Since $U\cap \hat{X}=\emptyset$, Lemma \ref{lem:inter-link} implies that either $I_{\hat{\mc{I}}}(U,p) \neq 0$ or $I_{\hat{\mc{I}}}(U,q) \neq 0$, so that either $p$ or $q$ is in $A$. This contradicts the fact that $A\subset B_{kM''}(z)$.
\end{proof}

As an immediate consequence we have the following 
\begin{corollary} Any $\hat{f}$-periodic free topological disk without wandering points is bounded (by a bound that depends only on the period).
\end{corollary}

\begin{proposition}\label{pro:inter-endpoint} If $U$ is an $\hat{f}$-periodic simply connected open set intersecting $\hat{X}$ and there are no wandering points of $\hat{f}$ in $U$, then every leaf of $\hat{\mc{F}}$ that intersects $\ol{U}$ has one endpoint in $U$.
\end{proposition} 
\begin{proof} Let $k$ be the smallest positive integer such that $\hat{f}^k(U)=U$, and let $z\in U\cap \hat{X}$. Since $\hat{f}^k(z)=z$, Proposition \ref{pro:link-periodic} implies that $I_{\hat{\mc{I}}}(U,p)=I_{\hat{\mc{I}}}(z,p)$ for any $p\in \hat{X}\sm U$. Since $z\in \hat{X}$ is fixed by the isotopy, it follows that $I_{\hat{\mc{I}}}(z,p)=0$ and therefore $I_{\hat{\mc{I}}}(U,p)=0$ for any $p\in \hat{X}\sm U$.

Suppose that a regular leaf $\Gamma$ of $\hat{\mc{F}}$ intersects $\ol{U}$, and let $p_1$ and $p_2$ be the endpoints of $\Gamma$. If neither $p_1$ nor $p_2$ is in $U$, then Lemma \ref{lem:inter-link} implies that $I_{\hat{\mc{I}}}(U,p_i)\neq 0$ for some $i\in \{1,2\}$, contradicting our previous claim. Therefore, one of the two endpoints of $\Gamma$ belongs to $U$.
\end{proof}

\section{A bound on invariant inessential open sets: Proof of Theorem \ref{th:bdfix}}
\label{sec:bdfix}

This section is devoted to the proof of
\begin{theorem*}[\ref{th:bdfix}] If $f\colon \T^2\to \T^2$ is a nonwandering non-annular homeomorphism homotopic to the identity then one and only one of the following properties hold:
\begin{itemize}
\item[(1)] There exists a constant $M$ such that each $f$-invariant open topological disk $U$ satisfies $\diamup(U)<M$; or
\item[(2)] $\fix(f)$ is fully essential and $f$ is irrotational.
\end{itemize}
\end{theorem*}

Let us outline the steps of the proof of Theorem \ref{th:bdfix}. First we use the fact that $f$ is non-annular to show that if $\fix(f)$ is essential, then it is fully essential, and case $(2)$ holds. Next, assuming the theorem does not hold, we show that it suffices to consider the case where $\fix(f)$ is totally disconnected, by collapsing the components of the filling of $\fix(f)$. For such $f$, and assuming that there are arbitrarily `large' invariant open topological disks, we show that there is a gradient-like Brouwer foliation associated to a lift $\hat{f}$ of $f$. Then we show that the invariant topological disks are bounded, as follows: if there is an unbounded invariant topological disk $U$, using the linking number defined in \S\ref{sec:linking}, and more specifically Proposition \ref{pro:inter-endpoint}, we have that every leaf of the foliation that intersects $\ol{U}$ has an endpoint in $U$. Using this fact and a geometric argument relying on the fact that $f$ is non-annular, we are able to conclude that the boundary of $U$ consists of singularities of the foliation (contradicting the fact that the set of singularities is totally disconnected). After this, we are able to obtain a sequence of pairwise disjoint bounded simply connected invariant sets with increasingly large diameter, and a variation of the previous argument leads again a contradiction.

\subsection{The case where $\fix(f)$ is essential}

\begin{proposition}\label{pro:fix-ess-irrotational} Under the hypotheses of Theorem \ref{th:bdfix}, suppose $\fix(f)$ is essential. Then $\fix(f)$ is fully essential, and there is a lift $\hat{f}$ of $f$ such that $\pi(\fix{\hat{f}})=\fix(f)$. Moreover, $\rho(\hat{f})=\{0\}$ (\ie $f$ is irrotational).
\end{proposition}
\begin{proof}
If $\fix(f)$ is essential but not fully essential, then $\T^2\sm \fix(f)$ is essential, and so it has some essential connected component $A$. The fact that $\fix(f)$ is essential implies that $A$ is not fully essential, and so it must be annular. Since connected components of $\T^2\sm \fix(f)$ are permuted by $f$, $A$ is a fixed annular set for $f^k$ for some $k>0$, and so $f^k$ is annular by Proposition \ref{pro:annular}. Moreover, since $f$ has a fixed point, by the same Proposition we conclude that $f$ is annular. This contradicts our hypothesis.

Thus $\fix(f)$ is fully essential, and there is some fully essential connected component $K$ of $\fix(f)$. Fix $z_0\in K$ and let $\hat{f}$ be a lift of $f$ such that $\hat{f}(\hat{z}_0)=\hat{z}_0$ for any $\hat{z}_0\in \pi^{-1}(z_0)$. We claim that $\pi^{-1}(K)\subset \fix(\hat{f})$. Indeed, the map defined by $z\mapsto \hat{f}(\hat{z})-\hat{z}$ for $\hat{z}\in \pi^{-1}(z)$ is well defined on $\T^2$ and continuous, and it takes integer values on $K$.  Since it is null at $z_0$ and $K$ is connected, it must be constantly zero on $K$. Thus $K\subset \pi(\fix(\hat{f}))$ (so $\fix(f)$ is fully essential).

Let us prove that $f$ is irrotational. Suppose for contradiction that $\rho(\hat{f})\neq \{0\}$. Then $\rho(\hat{f})$ has some nonzero extremal point $w$, and so by Proposition \ref{pro:rotation-set}, there is an $f$-ergodic Borel probability measure $\mu$ on $\T^2$ such that $\mu$-almost every point $x\in \T^2$ is such that, if $\hat{x}\in \pi^{-1}(x)$, then
$$\lim_{n\to \infty} \frac{\hat{f}^n(\hat{x})-\hat{x}}{n}= w.$$ 
By Poincar\'e recurrence, we may choose a recurrent $x\in \T^2$ such that the above condition holds. Let $\hat{x}\in \pi^{-1}(x)$ and let $U$ be the connected component of $\R^2\sm \fix(\hat{f})$ that contains $\hat{x}$. From Theorem \ref{th:brown-kister} we know that $\hat{f}(U)=U$. Moreover, since $\pi(U)$ is disjoint from $K$, which is fully essential, we have that $\pi(U)$ is inessential, so $\pi|_U$ is injective. 
Since $x$ is recurrent, there is a sequence $(n_k)_{k\in \N}$ of integers with $\lim_{k\to \infty} n_k=\infty$ such that $f^{n_k}(x)\to x$ as $k\to \infty$. Since $\pi|_U$ is injective, it conjugates $\hat{f}|_{U}$ to $f|_{\pi(U)}$.  In particular, $\hat{f}^{n_k}(\hat{x})\to \hat{x}$ as $k\to \infty$. But then $(\hat{f}^{n_k}(\hat{x})-\hat{x})/n_k\to 0\neq w$ as $k\to \infty$, contradicting our choice of $x$. This shows that $f$ is irrotational.

The claim that $\pi(\fix(\hat{f}))=\fix(f)$ follows from the fact that $f$ is irrotational.
\end{proof}

\begin{proposition} Under the hypotheses of Theorem \ref{th:bdfix}, if $\fix(f)$ is essential, then it is fully essential, and for each $M>0$ and $v\in \Z^2_*$ there is some connected component $U$ of $\T^2\sm \fix(f)$ such that $\diamup_v(U)>M$.
\end{proposition}
\begin{proof} The previous proposition implies that $\fix(f)$ is fully essential and $\fix(\hat{f})=\pi^{-1}(\fix(f))$. Each connected component of $\R^2\sm \fix(\hat{f})$ is $\hat{f}$-invariant. If there is a uniform bound on the diameter of such components, then one has a uniform bound on $|\hat{f}^n(z)-z|$ for $z\in \R^2$, $n\in \Z$, contradicting the fact that $f$ is non-annular.
\end{proof}

\subsection{The case where $\fix(f)$ is inessential}
\setcounter{claim}{0}
To complete the proof of Theorem \ref{th:bdfix} we will assume from now on that the theorem does not hold, and we will seek a contradiction. Thus we assume that there exists $f$ such that the hypotheses of the theorem hold but the thesis does not. The previous two propositions imply that $\fix{f}$ is essential if and only if case (2) of the theorem holds. Therefore, we may assume that $\fix{f}$ is inessential and item (1) does not hold. This means that for any $M$ there exists an open connected $f$-invariant topological disk $U$ such that $\diam(U)>M$. 

\subsection{Fixing a lift $\hat{f}$}

\begin{claim} There is a lift $\hat{f}$ of $f$ and a sequence $(U_n)_{n\in \Z}$ of open $\hat{f}$-invariant topological disks in $\R^2$ such that $\pi(U_n)$ is inessential and $\diam(U_n)\to \infty$ as $n\to \infty$.
\end{claim}

\begin{proof} There are $f$-invariant topological disks of arbitrarily large diameter, and each contains a fixed point of $f$ by Proposition \ref{pro:brouwer-trivial}. The claim follows from the fact that only finitely many lifts of ${f}$ may have fixed points.
\end{proof}

From now on we will work with the lift $\hat{f}$ and the sequence $(U_n)_{n\in \N}$ from the previous claim. 

\begin{claim}\label{claim:bdfix-nw} $U_n+v\subset \Omega(\hat{f})$ for all $n\in \N$ and $v\in \Z^2$.
\end{claim}
\begin{proof} Since $\pi(U_n+v) = \pi(U_n)$ is inessential, $\pi|_{U_n+v}$ is a homeomorphism onto its image which conjugates $\hat{f}|_{U_n+v}$ to $f|_{\pi(U_n)}$. Since the latter is nonwandering, so is $\hat{f}|_{U_n+v}$, implying that $U_n+v\subset \Omega(\hat{f})$.
\end{proof}

\subsection{Simplification of $\fix(f)$}

We will show that it is possible to assume that $\fix(f)$ is totally disconnected, by collapsing the connected components of $\fill(\fix(f))$ to points, while keeping all the hypothesis. To do so, we need to rule out the possibility that this process leads to a situation where there are no longer arbitrarily large simply connected sets.

\begin{claim} For each $M\in \R$ there is an open connected $\hat{f}$-invariant set $U\subset \R^2\sm \fix(\hat{f})$ such that $\pi(U)$ is inessential and $\diam(U)>M$.
\end{claim}
\begin{proof} Let $\mc{U}$ be the family of all open connected inessential subsets of $\T^2\sm \pi(\fix(\hat{f}))$ which are the projection of an $\hat{f}$-invariant subset of $\R^2$. We want to show that $\sup_{V\in  \mc{U}}\diamup(V)=\infty$. Suppose for contradiction that $\diamup(V)\leq M$ for all $V\in \mc{U}$. 

Since $\diam(U_n)\to \infty$, we may find $v\in \Z^2_*$ such that $\diam(P_v(U_n))\to \infty$, and since we are assuming that $\fix(f)$ is inessential, there is a simple loop $\gamma\subset \T^2\sm \fix(f)$ with homology class $v^\perp$, so that $\gamma$ lifts to an arc $\hat{\gamma}$ joining a point $z_0$ to $z_0+v^\perp$ and disjoint from $\fix(\hat{f})$. 
To simplify the notation, we will assume that $v = (1,0)$ and $[\hat{\gamma}] = \{0\}\times[0,1]$. The general case is analogous (in fact we can reduce the general case to this case by conjugating $f$ by an appropriately chosen homeomorphism).

We will show that for any given $m>0$ we can find $m$ pairwise disjoint subarcs $\gamma_1,\dots,\gamma_m$ of $\gamma$ such that $[\gamma_i]\cap f([\gamma_i])\neq \emptyset$ for each $i\in \{1,\dots, m\}$. This is enough to complete the proof of the claim, because it leads to a contradiction as follows: Since $f$ has no fixed point in $\gamma$, we may choose $m$ such that $d(f(x),x))>1/m$ for each $x\in [\gamma]$. Since the arcs $\gamma_i$ are  pairwise disjoint and $\gamma$ has length $1$ (because we are assuming it is a vertical circle), one of the arcs $\gamma_i$ has diameter at most $1/m$. Since $f(\gamma_i)$ intersects $\gamma_i$, it follows that there is a point $x\in [\gamma_i]$ such that $d(f(x),x)\leq 1/m$, contradicting our choice of $m$. This contradiction completes the proof, assuming the existence of the arcs $\gamma_i$. We devote the rest of the proof to prove the existence of such arcs. 

Let $N_0\in \N$ be such that $N_0>M$, and denote $\Gamma=\{0\}\times \R$. If $m\in \N$ is fixed and $n$ is chosen large enough, then there is $i_0$ such that $U_n$ intersects $\Gamma+(N_0i,0)$ for each $i\in \{i_0,i_0+1,\dots, i_0+m+1\}$. Fix $i\in \N$ with $i_0< i \leq i_0+m$, and let $p_1\in U_n\cap (\Gamma+(N_0(i-1),0))$ and $p_2\in U_n\cap(\Gamma+(N_0(i+1),0))$. Then $p_1$ and $p_2$ are in different connected components of $U_n\sm (\Gamma+(N_0i,0))$. From this and from the fact that $U_n$ is a topological disk, it is easy to verify that there is a connected component $\hat{\gamma}_i$ of $U_n\cap (\Gamma+(N_0i,0))$ that separates $p_1$ from $p_2$ in $U_n$; that is, $U_n\sm [\hat{\gamma}_i]$ contains $p_1$ and $p_2$ in different connected components $V_1$ and $V_2$, respectively. Since $\hat{\gamma_i}$ is a cross-cut of $U_n$ (\ie a simple arc in $U_n$ joining two points of its boundary), we have $U_n\sm [\hat{\gamma}_i] = V_1\cup V_2$. Since $V_1\subset U_n$ intersects $\Gamma+(N_0(i-1),0)$ and has a point of $\Gamma+(N_0i,0)$ in its boundary, it follows that $\diam(V_1)\geq N_0>M$. Because of this, $V_1$ cannot be contained in $U_n\sm \fix(\hat{f})$: otherwise, the connected component of $U_n\sm \fix(\hat{f})$ that contains $V_1$ would be an element of $\mc{U}$ (since Theorem \ref{th:brown-kister} implies that it is $\hat{f}$-invariant), contradicting our assumption that $\diam(V)\leq M$ for all $V\in \mc{U}$. Hence, $V_1$ contains a fixed point of $\hat{f}$. Similarly, since $V_2\subset U_n$ intersects $\Gamma+(N_0(i+1),0)$ and its boundary intersects $\Gamma+(N_0i,0)$, we have $\diam(V_2)\geq N_0 > M$ and we conclude in the same way that $V_2$ contains a fixed point of $\hat{f}$.

Therefore $\hat{f}(V_1)\cap V_1\neq \emptyset$ and $\hat{f}(V_2)\cap V_2\neq \emptyset$, and since $V_1\cup V_2 = U_n\sm [\hat{\gamma}_i]$ and $\hat{f}|_{U_n}$ is nonwandering, it follows from these facts that $\hat{f}([\hat{\gamma}_i])\cap [\hat{\gamma}_i]\neq \emptyset$. To complete the proof, observe that the arcs $\hat{\gamma}_{i_0+1}, \dots, \hat{\gamma}_{i_0+m}$ thus obtained project to pairwise disjoint subarcs $\gamma_1,\dots, \gamma_m$ of $\gamma$, because they are pairwise disjoint subarcs of $\pi^{-1}([\gamma])\cap U_n$, and $U_n$ projects to $\T^2$ injectively. 
\end{proof}

\begin{claim}\label{claim:bdfix13} We may assume that $\fix(f)$ is totally disconnected.
\end{claim}
\begin{proof} 
The previous claim implies that there exists a sequence $(\hat{V}_n)_{n\in \N}$ of open connected $\hat{f}$-invariant subsets of $\R^2\sm \fix(\hat{f})$ such that $\diam(\hat{V}_n)\to \infty$ as $n\to \infty$ and $V_n=\pi(\hat{V}_n)$ is inessential for each $n\in \N$. This implies that $V_n\subset \T^2\sm \fix(f)$, because the fact that $V_n$ projects injectively implies that any element of $\pi^{-1}(\fix(f))\cap V_n$ must be a fixed point of $\hat{f}$.

 Since $\fix(f)$ is inessential, so is $K=\fill(\fix(f))$. Moreover, by Proposition \ref{pro:inessential-bound} there is a uniform bound on $\diamup(C)$ among the connected components $C$ of $K$. Since $\diamup(V_n)\to \infty$, this implies that there is $n_0$ such that $V_n\subset \T^2\sm K$ if $n\geq n_0$. 

Proposition \ref{pro:collapse} implies that there is a continuous surjection $h\colon \T^2\to \T^2$ homotopic to the identity and a homeomorphism $f'\colon \T^2\to \T^2$ such that $hf = f'h$, and additionally $h(K)$ is totally disconnected ($h$ collapses components of $K$ to points) and $h|_{\T^2\sm K}$ is a homeomorphism onto $\T^2\sm h(K)$. Furthermore, since every component of $K$ contains a fixed point of $f$, and there are no fixed points outside $K$, it follows that $h(K)=\fix(f')$. The map $f'$ is clearly nonwandering, and the sets $(h(V_n))_{n\geq n_0}$ provide a sequence of simply connected open $f'$-invariant subsets of $\T^2\sm \fix(f')$. Moreover, since $h$ is homotopic to the identity, if $\hat{h}\colon \R^2\to \R^2$ is a lift of $h$ then there is a constant $M'$ such that $\norm{\smash{\hat{h}(x)-x}}<M'$ for all $x\in \R^2$. If $\hat{V}_n$ is a connected component of $\pi^{-1}(V_n)$, then $\hat{h}(\hat{V}_n)$ is a connected component of $\pi^{-1}(h(V_n))$ and 
$$\diamup(h(V_n))=\diam(\hat{h}(\hat{V}_n))\geq \diam(\hat{V}_n) -2M = \diamup(V_n)-2M \xrightarrow[n\to \infty]{} \infty.$$ 

Hence $\diamup(\fill(h(V_n)))\to \infty$ as $n\to \infty$, and since $\fill(h(V_n))$ is an $f'$-invariant topological disk, we have that $f'$ satisfies the hypotheses but not the thesis of the theorem. Thus, by working with $f'$ instead of $f$ since the beginning of the proof of the theorem, we may have assumed that $\fix(f)$ is totally disconnected.
\end{proof}

\subsection{Unboundedness in every direction of the sets $U_n$}

\begin{claim}\label{claim:bdfix5} $\diam(\proj_v(U_n)) \to \infty$ as $n\to \infty$ for each $v\in \Z^2_*$.
\end{claim}
\begin{proof} 
Suppose the claim does not hold. Then, after passing to a subsequence of $(U_n)_{n\in \N}$, we may assume that there is $v\in \Z^2_*$ and a constant $M$ such that $\diam(\proj_v(\ol{U_n}))\leq M$ for all $n\in \N$.  
Let $$A=\ol{\bigcup_{k\in \Z}\bigcup_{n\in \N} U_n+kv^\perp}.$$ 
The fact that $\diam(U_n)\to \infty$ implies that the sets
$$V^-=\proj_v^{-1}((-\infty,-M))\text { and } V^+=\proj_v^{-1}((M, \infty))$$ are contained in different connected components of $\R^2\sm A$. Let us call these components $\til{V}^+$ and $\til{V}^-$, respectively. 

Note that since each $U_n$ is $\hat{f}$-invariant, we have $\hat{f}(A)=A$, and so the connected components of $\R^2\sm A$ are permuted by $\hat{f}$. The fact that $\hat{f}(x)-x$ is uniformly bounded implies that $\hat{f}(\til{V}^+)=\til{V}^+$ and $\hat{f}(\til{V}^-)=\til{V}^-$.  Since $V^-\subset \til{V}^-$ and $\til{V}^-$ is disjoint from $\til{V}^+\supset V^+$, we have  $$\proj_v^{-1}((-\infty,-M'))\subset \til{V}^-\subset \proj_v^{-1}((-\infty, M')),$$
and we conclude from Proposition \ref{pro:wall-annular} that $f$ is annular. This contradicts the hypothesis of the theorem, proving the claim.
\end{proof}

\subsection{Maximality and disjointness of $U_n$}

\begin{claim}\label{claim:bdfix7} We may assume that each $U_n$ is maximal with respect to the property of $\pi(U_n)$  being open, $f$-invariant and simply connected (\ie $U_n$ is not properly contained in a set with the same properties). 
\end{claim}
\begin{proof}
By a direct application of Zorn's Lemma, there exists an open simply connected $f$-invariant set $\til{U}_n'\subset \T^2$ such that $\pi(U_n)\subset \til{U}_n'$ and $\til{U}_n'$ is maximal with the property of being open, $f$-invariant, and simply connected. The connected component $\til{U}_n$ of $\pi^{-1}(\til{U}_n')$ that contains $U_n$ satisfies the required properties, so we may replace $U_n$ by $\til{U}_n$ for each $n\in \N$.
\end{proof}

\begin{claim} If $U_n$ and $U_m$ are bounded and $\pi(U_n)\cap \pi(U_m) \neq \emptyset$ then $\pi(U_n)=\pi(U_m)$.
\end{claim}
\begin{proof}
If $\pi(U_n)\cap \pi(U_m)\neq \emptyset$, then there exists $w\in \Z^2$ such that $U_n\cap (U_m+w)\neq \emptyset$. Let $U=\fill(U_n\cup (U_m+w))$, which is bounded and $\hat{f}$-invariant. Suppose that $\pi(U)$ is essential. Then there is $v\in \Z^2_*$ such that $U\cap (U+v)\neq \emptyset$. Let $V= \bigcup_{k\in \Z} U+kv$. The fact that $U$ is bounded implies that $\diam \proj_{v^\perp}(V)<\infty$. However, $V$ is at the same time $\hat{f}$-invariant. Similar to previous cases, an application of Proposition \ref{pro:wall-annular} now shows that $\hat{f}$ is annular, contradicting the hypothesis of the theorem. Thus $\pi(U)$ is inessential,
and since $U$ is filled and connected, $\pi(U)$ is an open $f$-invariant topological disk which contains $\pi(U_n)$ and $\pi(U_m)$. It follows form the maximality of $\pi(U_n)$ and $\pi(U_m)$ that $\pi(U_n)=\pi(U)=\pi(U_m)$, as claimed.
\end{proof}

\begin{claim}\label{claim:bdfix14} We may assume that the disks $(\pi(U_n))_{n\in \N}$ are either pairwise disjoint or all equal to $\pi(U_0)$.
\end{claim}
\begin{proof} Assume first that $U_n$ is unbounded for some $n$. Then we may assume that $U_m=U_n$ for each $m\in \N$, and all the required hypotheses hold. Now assume that each $U_n$ is bounded. Since by our hypothesis $\diam(U_n)\to \infty$ as $n\to \infty$, for each $n\in \N$ we may find $m\in \N$ such that $\diam(U_m)>\diam(U_k)$ for all $k\in \{1,\dots,n\}$, so that $\pi(U_m)\neq \pi(U_k)$ for all $k\in \{1,\dots,n\}$. Using this fact, we may extract a subsequence of $(U_n)_{n\in \N}$ which projects to pairwise distinct disks, and these disks must be pairwise disjoint due to the preivous claim.
\end{proof}

\subsection{Obtaining a gradient-like Brouwer foliation}
Since from Claim \ref{claim:bdfix13} we are assuming that $\fix(f)$ is totally disconnected, by Proposition \ref{pro:brouwer} we know that there is an oriented foliation $\mc{F}$ with singularities $X=\sing(\mc{F})$ and an isotopy $\mc{I}=(f_t)_{t\in [0,1]}$ from the identity to $f$ fixing $X$ pointwise, such that $\mc{I}$ lifts to the isotopy $\hat{\mc{I}}=(\hat{f}_t)_{t\in [0,1]}$ from the identity to $\hat{f}$ and such that for any $z\in \T^2\sm X$ the arc $(f_t(z))_{t\in [0,1]}$ is homotopic with fixed endpoints in $\T^2\sm X$ to a positively transverse (to $\mc{F}$) arc. The set $\hat{X} = \pi^{-1}(\hat{X})\subset \fix(\hat{f})$ is the set of singularities of $\hat{\mc{F}}$ and is fixed by $\hat{\mc{I}}$. For each $z\in \R^2\sm \hat{X}$, the arc $(\hat{f}_t(z))_{t\in [0,1]}$ is homotopic with fixed endpoints in $\R^2\sm \hat{X}$ to a positively transverse (to $\hat{\mc{F}}$) arc.

Our purpose now is to apply Proposition \ref{pro:zerohull} to show that $\mc{F}$ is gradient-like. In view of Remark \ref{rem:zerohull}, it suffices to find $v_1,v_2,v_3,v_4\in \Z^2_*$ such that $0$ is in the interior of the convex hull of $\{v_1,v_2,v_3,v_4\}$ and for each $i$ a positively transverse arc $\gamma_i$ in $\R^2\sm \hat{X}$ such that $\gamma_i(1)-\gamma_i(0)=v_i$. Indeed this would imply that $\mc{C}(z)$ contains $0$ in the interior of its convex hull for some $z\in \T^2$ and thus that $\mc{F}$ is gradient-like. 

\begin{claim} For each $v\in \Z^2_*$ there is $x\in \R^2\sm \hat{X}$ and $w\in \Z^2_* \sm \R v$ such that there are positively transverse arcs from $x$ to $x+w$ and from $x$ to $x-w$.
\end{claim}
\begin{proof} 
Since $\R^2\sm \hat{X}$ is connected, for any $z\in \R^2\sm \hat{X}$ we may find an arc $\gamma$ joining $z$ to $z+v$ in $\R^2\sm \hat{X}$. Let $\Gamma = \bigcup_{n\in \Z} [\gamma]+n{v}$. By Claim \ref{claim:bdfix5}, $\proj_{v^\perp}(U_n)\to \infty$. In particular, given $m\in \N$ there is $n=n_m$ such that $\proj_{v^\perp}(U_n)>\diam(\proj_{v^\perp}(\Gamma))+(m+1)\norm{v^\perp}$. It follows that $U_n+iv^\perp$ intersects $\Gamma$ for at least $m$ consecutive values of $i$. Thus there is a set $\{(i_1, j_1)$, \dots, $(i_m, j_m)\}\subset \Z^2$ such that $U_n+i_kv^\perp+j_kv$ intersects $[\gamma]$ for $1\leq k\leq m$ and $i_k = i_0+k$ for some $i_0\in \Z$. Choosing one point in each intersection $[\gamma]\cap (U_n+i_kv^\perp+j_kv)$, we get $m$ points in $[\gamma]$, and so by a pigeonhole argument two of them must be a distance less than $r_m = \sqrt{2}\diam[\gamma]/\floor{\sqrt{m}}$ apart (where $\floor{y}$ is the largest integer not greater than $y$). 
Thus one can find $x_m\in [\gamma]$ such that $B_{r_m}(x_m)$ intersects $U_n+i_kv^\perp+j_kv$ for two different values of $k$. Note that this implies that $B_{r_m}(x_m)$ intersects $U_n+u$ and $U_n+u'$ for two different elements $u,u'\in \Z^2$ such that $u'-u\notin\R v$ (because $i_k\neq i_{k'}$ if $k\neq k'$).

Letting $x$ be a limit point of $(x_m)_{m\in \N}$ one sees that for any $r>0$, there are arbitrarily large values of $n$ for which there are at least two different elements $u,u'\in \Z^2$ such that $B_r(x)$ intersects both $U_n+u$ and $U_n+u'$, and $u'-u\notin \R v$.

In particular, since $U_n+u\subset \Omega(\hat{f})$ for all $u\in \Z^2$ and $n\in \N$, this implies that $x$ is nonwandering for $\hat{f}$, so by Proposition \ref{pro:pta} there is a neighborhood $V_x$ of $x$ such that every point of $V_x$ can be joined to any other point of $V_x$ with a positively transverse arc. But then we can find $n\in \Z$ and $u,u'\in \Z^2$  such that $w=u'-u\notin \R v$ and $V_x$ intersects both $U_n+u$ and $U_n+u'$, so that $U_n+u'$ intersects both $V_x$ and $V_x+w$. If $z_0\in V_x\cap (U_n+u')$ and $z_1\in (V_x+w)\cap (U_n+u')$ then we can find a positively transverse arc $\sigma$ from $x$ to $z_1$ because of our choice of $V_x$, and we may find a positively transverse arc $\alpha$ from $z_0\in U_n+u'$ to $z_1 \in U_n+u'$ because $U_n+u'$ is connected and contained $\Omega(\hat{f})$ (see Proposition \ref{pro:pta}). Finally, we may find a positively transverse arc $\eta$ from $z_1-w\in V_x$ to $x$, so that $\eta+w$ is a positively transverse arc from $z_1$ to $x+(u'-u)$. Therefore $\sigma*\alpha*(\eta+w)$ is a positively transverse arc $\hat{\alpha}$ from $x$ to $x+w$. The same argument can be repeated in the opposite direction, obtaining a positively transverse arc $\hat{\beta}$ from $x+w$ back to $x$, which translated by $-w$ provides a positively transverse arc from $x$ to $x-w$. This proves the claim.
\end{proof}

As explained before its statement, the previous claim together with Proposition \ref{pro:zerohull} allow to conclude the following:
\begin{claim} The foliation $\mc{F}$ is gradient-like.
\end{claim}

\subsection{Linking of the sets $U_n$ and points of $\hat{X}$}

Since $\mc{F}$ is gradient-like, every regular leaf $\Gamma$ of $\hat{\mc{F}}$ connects two different elements of $\hat{X}$ and there is a uniform bound $\diam(\Gamma)<M_0$.

\begin{claim} We may assume that $U_n\cap \hat{X}\neq \emptyset$ for each $n\in \N$.
\end{claim}

\begin{proof} By Claim \ref{claim:bdfix-nw}, $U_n$ has no wandering points. By Proposition \ref{pro:large-X} there is $M_1$ such that if $\diam(U_n)>M_1$, then $U_n\cap \hat{X}\neq \emptyset$. Since $\diam(U_n)\to \infty$, by extracting a subsequence and re-indexing, we may assume that  $U_n\cap \hat{X}\neq \emptyset$ for all $n$.
\end{proof}

\begin{claim}\label{claim:bdfix2} For any $n\in \N$ and $v\in \Z^2$, every regular leaf of $\hat{\mc{F}}$ that intersects $\cl(U_n+v)$ has one endpoint in $U_n+v$.
\end{claim}
\begin{proof}
Since $U_n$ intersects $\hat{X}$ and $\hat{X}$ is $\Z^2$-invariant, it follows that $U_n+v$ intersects $\hat{X}$, and the claim follows from Proposition \ref{pro:inter-endpoint} (recalling that $f|_{U_n+v}$ is nonwandering by Claim \ref{claim:bdfix-nw}).
\end{proof}

\subsection{Boundedness of $U_n$}

\begin{claim}\label{claim:bdfix6} $U_n$ is bounded for each $n\in \N$.
\end{claim}

Suppose for contradiction that $U=U_n$ is unbounded for some $n$. Then we may have assumed from the beginning of the proof of the theorem that $U=U_n$ for all $n$, since the hypotheses hold for that case. In particular, Claim \ref{claim:bdfix5} implies that $\diam \proj_v(U)=\infty$ for any $v\in \Z^2_*$.
From now until the end of this subsection, we seek a contradiction to prove Claim \ref{claim:bdfix6}. 

Let $W$ be the union of $U$ with all leaves of $\hat{\mc{F}}$ that intersect $U$. Observe that $W$ is open.

\begin{claim} $W\cap (W+v)=\emptyset$ for each $v\in \Z^2_*$. 
\end{claim}

\begin{proof}
Suppose for contradiction that this is not the case. Then some leaf $\Gamma$ of $\hat{\mc{F}}$ intersects both $U$ and $U+v$. Thus $\Gamma$ joins a point $p\in U\cap \hat{X}$ to a point $q\in (U\cap \hat{X})+v$, where $0\neq v\in \Z^2$. Let $\gamma$ be the subarc of $\Gamma$ joining $p$ to $q$, and let $\sigma$ be any arc in $U+v$ joining $q$ to $p+v$. Then $\gamma*\sigma$ is an arc joining $p$ to $p+v$. Hence $\Theta=\bigcup_{n\in \Z} [\gamma*\sigma]+nv$ is a closed connected set that separates $\R^2$, and $\proj_{v^\perp}(\Theta)$ is bounded.
The fact that $\diam \proj_{v^\perp}(U)=\infty$ implies that $U+mv^\perp$ intersects $\Theta$ for some $m\in \Z$, $m\neq 0$, and so $U+mv^\perp$ intersects $[\gamma*\sigma]+nv$ for some $n\in \Z$. But $U+mv^\perp$ is disjoint from $[\sigma]+nv \subset U+(n+1)v$, since otherwise $U+mv^\perp-(n+1)v$ would intersect $U$, contradicting the fact that $\pi(U)$ is inessential (noting that $mv^\perp-(n+1)v\neq 0$, since $m\neq 0$). Therefore $U+mv^\perp$ intersects $[\gamma]+nv$. But since $\Gamma+nv$ is a leaf of $\hat{\mc{F}}$ and it contains $\gamma+nv$, it follows from Claim \ref{claim:bdfix2} that $\Gamma+nv$ has one endpoint in $U+mv^\perp$. On the other hand we know that the endpoints of $\Gamma+nv$ are $p+nv\in U+nv$ and $q+nv \in U+(n+1)v$, both of which are disjoint from $U+mv^\perp$ (since $m\neq 0$). This contradiction shows that $W\cap (W+v)=\emptyset$ for each $v\in \Z^2_*$.

\end{proof}

Now let $\mc{O}=\bigcup_{n\in \Z} \hat{f}^n(W)$. Note that  $\mc{O}$ is open, connected and $\hat{f}$-invariant. 

\begin{claim} $\mc{O}\cap (U+v)=\emptyset$ for each $v\in \Z^2_*$
\end{claim}
\begin{proof}
Indeed, since $W\cap (W+v)=\emptyset$, in particular $W\cap (U+v)=\emptyset$. Since $U+v$ is invariant, it follows that $\hat{f}^k(W)\cap (U+v)=\emptyset$ for any $k\in \Z$, and the claim follows.
\end{proof}

\begin{claim} $\mc{O}\cap (\mc{O}+v)=\emptyset$ for each $v\in \Z^2_*$ (\ie $\pi(\mc{O})$ is inessential).
\end{claim}
\begin{proof} If $\mc{O}\cap (\mc{O}+v)\neq \emptyset$ and $v\neq 0$, since $\mc{O}$ is connected it contains an arc $\sigma$ joining some point $z\in \mc{O}$ to $z+v\in \mc{O}$. If $\Theta=\bigcup_{n\in \Z} [\sigma]+nv$, then $\Theta$ is bounded in the $v^\perp$ direction, and the fact that $\diam(P_{v^\perp}(U))=\infty$ implies that $U+mv^\perp$ intersects $\Theta$ for some $m\in \Z$, $m\neq 0$. But then $U+mv^\perp$ intersects $[\sigma]+nv$ for some $n\in \Z$, so that $U+mv^\perp -nv$ intersects $[\sigma]\subset \mc{O}$. Since $mv^\perp-nv\neq 0$, this contradicts the previous claim.
\end{proof}

\begin{claim}\label{claim:bdfix8} If a regular leaf $\Gamma$ of $\hat{\mc{F}}$ intersects $U$, then it is contained in $U$.
\end{claim}
\begin{proof} 
Let $\til{\mc{O}} = \fill(\mc{O})$. It follows from the properties of $\mc{O}$ that $\til{\mc{O}}$ is simply connected, $\hat{f}$-invariant, and $\pi(\til{\mc{O}})$ is still inessential.
Thus $\til{\mc{O}}$ is a simply connected open $f$-invariant set that projects to an inessential set. Recalling that we are assuming (since Claim \ref{claim:bdfix7}) the maximality of $U$ with respect to these properties, and since $U\subset \til{\mc{O}}$, we conclude that $\til{\mc{O}}=U$. If a leaf $\Gamma$ of $\hat{\mc{F}}$ intersects $U$, then by the definition of $W$ and $\til{\mc{O}}$ it follows that $[\Gamma]\subset W\subset \mc{O}\subset \til{\mc{O}} = U$, proving our claim.
\end{proof}

\begin{claim}\label{claim:bdfix9} $\bd U\subset \hat{X}$
\end{claim}
\begin{proof} If this is not the case, there is some regular leaf $\Gamma$ of $\hat{\mc{F}}$ such that $[\Gamma]\cap \bd U\neq \emptyset$. But Claim \ref{claim:bdfix2} implies that $\Gamma$ has one endpoint in $U$, and thus by our previous claim $\Gamma$ is entirely in $U$, a contradiction. 
\end{proof}

The last claim is the sought contradiction: since $\hat{X}$ is totally disconnected, it cannot contain the boundary of a topological disk. This contradiction completes the proof of Claim \ref{claim:bdfix6}, $\ie$ that $U_n$ is bounded for each $n$.

\subsection{End of the proof of Theorem \ref{th:bdfix}}

Now that we know that each $U_n$ is bounded, Claim \ref{claim:bdfix14} implies that the sets  $\pi(U_n)_{n\in \N}$ are pairwise disjoint. To finish the proof, we will repeat the same arguments from the proof of Claim \ref{claim:bdfix6}, however with the difference that now the sets $U_n$ are bounded, so the proofs of the claims change (note that we could not have used these arguments prior to knowing that the sets $U_n$ are pairwise disjoint, for which we needed to know that $U_n$ is bounded).

Recall that we are assuming that
\begin{itemize}
\item The sets $(U_n)_{n\in \N}$ are maximal in the sense of Claim \ref{claim:bdfix7}.
\item The sets $(\pi(U_n))_{n\in \N}$ are pairwise disjoint.
\item $U_n\cap \hat{X}\neq \emptyset$ for each $n\in \N$.
\end{itemize}

Let $W$ be the union of $U_1$ with all leaves of $\hat{\mc{F}}$ that intersect $U_1$. Observe that $W$ is open.

\begin{claim} $W\cap (W+v)=\emptyset$ for each $v\in \Z^2_*$. 
\end{claim}

\begin{proof}
Suppose for contradiction that this is not the case. Then some leaf $\Gamma$ of $\hat{\mc{F}}$ intersects both $U_1$ and $U_1+v$. Thus $\Gamma$ joins a point $p\in U_1\cap \hat{X}$ to a point $q\in (U_1\cap \hat{X})+v$, where $0\neq v\in \Z^2$. Let $\gamma$ be the subarc of $\Gamma$ joining $p$ to $q$, and let $\sigma$ be any arc in $U_1+v$ joining $q$ to $p+v$. Then $\gamma*\sigma$ is an arc joining $p$ to $p+v$. Letting $\Theta=\bigcup_{n\in \Z} [\gamma*\sigma]+nv$, the fact that $\diam \proj_{v^\perp}(U_n)\to \infty$ as $n\to \infty$ implies that for any sufficiently large $n$, there is $w$ such that $U_n+w$ intersects $\Theta$, and so $U_n+w$ intersects $[\gamma*\sigma]+kv$ for some $k\in \Z$. But $U_n+w$ is disjoint from $[\sigma]+kv \subset U_1+(k+1)v$, since otherwise $U_1+w-(k+1)v$ would intersect $U_n$, contradicting the fact that $\pi(U_1)\neq \pi(U_n)$. Therefore $U_n+w$ intersects $[\gamma]+kv$. But since $\Gamma+kv$ is a leaf of $\hat{\mc{F}}$ and it contains $\gamma+kv$, it follows from Claim \ref{claim:bdfix2} that $\Gamma+kv$ has one endpoint in $U_n+w$. On the other hand we know that the endpoints of $\Gamma$ are $p+kv\in U_1+kv$ and $q+kv \in U_1+(k+1)v$, both of which are disjoint from $U_n+w$ (again, because $\pi(U_1)\neq \pi(U_n)$). This contradiction shows that $W\cap (W+v)=\emptyset$ for each $v\in \Z^2_*$.
\end{proof}

Now let $\mc{O}=\bigcup_{n\in \Z} \hat{f}^n(W)$. Note that  $\mc{O}$ is open, connected and $\hat{f}$-invariant. 

\begin{claim} $\mc{O}\cap (U_1+v)=\emptyset$ for each $v\in \Z^2_*$
\end{claim}
\begin{proof}
Indeed, since $W\cap (W+v)=\emptyset$, in particular $W\cap (U_1+v)=\emptyset$. Since $U_1+v$ is invariant, it follows that $\hat{f}^k(W)\cap (U_1+v)=\emptyset$ for any $k\in \Z$, and the claim follows.
\end{proof}

\begin{claim} $\mc{O}\cap (\mc{O}+v)=\emptyset$ for each $v\in \Z^2_*$ (\ie $\pi(\mc{O})$ is inessential).
\end{claim}

\begin{figure}[ht]
\centering
\includegraphics[width = \textwidth]{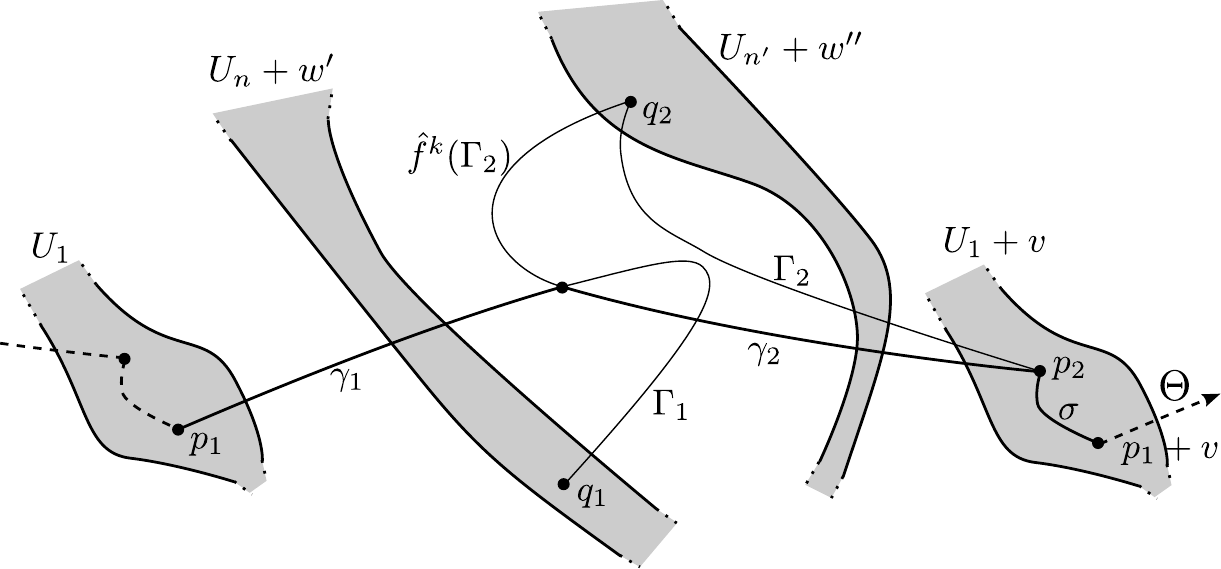}
\caption{$U_n$ has a translate intersecting $\Theta$ for at most two $n's$.}
\label{fig:claimO}
\end{figure}
\begin{proof} If $\mc{O}\cap (\mc{O}+v)\neq \emptyset$ and $v\neq 0$, then $W\cap (\mc{O}+v)\neq \emptyset$, This means that there are leaves $\Gamma_1$ and $\Gamma_2$ of $\hat{\mc{F}}$ such that $\Gamma_1$ has endpoints $p_1,q_1\in \hat{X}$ with $p_1\in U_1$, and $\Gamma_2$ has endpoints $p_2,q_2\in \hat{X}$ with $p_2\in U_1+v$, and there is an integer $k$ such that $f^k(\Gamma_2)\cap \Gamma_1\neq \emptyset$. Let $\gamma_1$ be a subarc of $\Gamma_1$ from $p_1$ to some intersection point $z\in f^k(\Gamma_2)\cap \Gamma_1$, and $\gamma_2$ a subarc of $f^k(\Gamma_2)$ from $z$ to $p_2$. Finally let $\sigma$ be an arc in $U_1+v$ joining $p_2$ to $p_1+v$. 

Like in the previous arguments, the fact that $\diam(P_{v^\perp}(U_n))\to \infty$ as $n\to \infty$ implies that when $n$ is large enough, there is $w\in \Z^2$ such that $U_n+w$ intersects $[\gamma_1*\gamma_2*\sigma]+mv$ for some $m\in \Z$, and as before, the fact that $\pi(U_n)\neq \pi(U_1)$ implies that $U_n+w$ is disjoint from $[\sigma]+mv$, so $U_n+w$ intersects either $[\gamma_1]+mv$ or $[\gamma_2]+mv$. This means that, if $w'=w-mv$, then $U_n+w'$ intersects either $\Gamma_1$ or $f^k(\Gamma_2)$. But since $U_n+w'$ is invariant, this implies that $U_n+w'$ intersects either $\Gamma_1$ or $\Gamma_2$. Again, Claim \ref{claim:bdfix2} implies that one of the endpoints of $\Gamma_1$ or $\Gamma_2$ is in $U_n+w'$. Since $\pi(p_1)\in \pi(U_1)\notin \pi(U_n)$ and similarly $\pi(p_2)\notin \pi(U_n)$, it follows that one of the points $q_1$ or $q_2$ is in $U_n+w'$. Without loss of generality, suppose that $q_1\in U_n+w'$.

Repeating the previous argument with another (sufficiently large) integer $n'> n$, we conclude that one of the points $q_1$ or $q_2$ is in $U_{n'}+w''$ for some $w''\in \Z^2$. Since $q_1\in U_n+w'$ which is disjoint from $U_{n'}+w''$ (because $\pi(U_{n'})\neq \pi(U_n)$) we conclude that $q_2\in U_{n'}+w''$. See Figure \ref{fig:claimO}.

But repeating this argument a third time, with some $n''>n'$, we conclude that there is $w'''\in \Z^2$ such that $U_{n''}+w'''$ contains $q_1$ or $q_2$, and this is not possible since $U_{n''}$ is disjoint from $U_{n'}+w''$ and $U_n+w'$.
This contradiction proves the claim.

\end{proof}

The next two steps are proved identically to Claims \ref{claim:bdfix8} and \ref{claim:bdfix9}.
\begin{claim} If a regular leaf $\Gamma$ of $\hat{\mc{F}}$ intersects $U_1$, then it is contained in $U_1$.
\end{claim}

\begin{claim} $\bd U_1\subset \hat{X}$
\end{claim}

Again, the last claim is a contradiction, since $\hat{X}$ is totally disconnected.
This concludes the proof of Theorem \ref{th:bdfix}.

\bibliographystyle{koro} 
\bibliography{essential}

\newcommand{\etalchar}[1]{$^{#1}$}
\providecommand{\bysame}{\leavevmode\hbox to3em{\hrulefill}\thinspace}
\providecommand{\MR}{\relax\ifhmode\unskip\space\fi MR }
\providecommand{\MRhref}[2]{%
  \href{http://www.ams.org/mathscinet-getitem?mr=#1}{#2}
}
\providecommand{\href}[2]{#2}
\begin{thebibliography}{LKFA90}

\bibitem[BK84]{brown-invariant}
M.~Brown and J.~M. Kister, \emph{Invariance of complementary domains of a fixed
  point set}, Proc. Amer. Math. Soc. \textbf{91} (1984), no.~3, 503--504.

\bibitem[Bro71]{brown-nielsen}
R.~Brown, \emph{The {L}efschetz fixed point theorem}, Scott Foresman and Co.,
  1971.

\bibitem[Chi79]{x_1}
B.~V. Chirikov, \emph{A universal instability of many-dimensional oscillator
  systems}, Phys. Rep. \textbf{52} (1979), no.~5, 264--379.

\bibitem[Dav86]{daverman}
R.~J. Daverman, \emph{Decompositions of manifolds}, Pure and Applied
  Mathematics, vol. 124, Academic Press Inc., Orlando, FL, 1986.

\bibitem[D{\'a}v13]{davalos}
P.~D{\'a}valos, \emph{On torus homeomorphisms whose rotation set is an
  interval}, Math.~Z. (2013), doi:10.1007/s00209-013-1168-3.

\bibitem[Eps66]{epstein}
D.~B.~A. Epstein, \emph{Curves on $2$-manifolds and isotopies}, Acta
  Mathematica \textbf{115} (1966), 83--107.

\bibitem[Fat87]{fathi}
A.~Fathi, \emph{An orbit closing proof of {B}rouwer's lemma on translation
  arcs}, Enseign. Math. (2) \textbf{33} (1987), no.~3-4, 315--322.

\bibitem[Fra88]{franks-reali2}
J.~Franks, \emph{Recurrence and fixed points of surface homeomorphisms},
  Ergodic Theory Dynam. Systems \textbf{8$^*$} (1988), no.~Charles Conley
  Memorial Issue, 99--107.

\bibitem[Fra89]{franks-reali}
\bysame, \emph{Realizing rotation vectors for torus homeomorphisms}, Trans.
  Amer. Math. Soc. \textbf{311} (1989), no.~1, 107--115.

\bibitem[Fra95]{franks-reali3}
\bysame, \emph{The rotation set and periodic points for torus homeomorphisms},
  Dynamical systems and chaos, {V}ol. 1 ({H}achioji, 1994), World Sci. Publ.,
  River Edge, NJ, 1995, pp.~41--48.

\bibitem[Fur61]{furstenberg}
H.~Furstenberg, \emph{Strict ergodicity and transformation of the torus}, Amer.
  J. Math. \textbf{83} (1961), 573--601.

\bibitem[Gut79]{gutierrez}
C.~Guti{\'e}rrez, \emph{Smoothing continuous flows and the converse of
  {D}enjoy-{S}chwartz theorem}, An. Acad. Brasil. Ci\^enc. \textbf{51} (1979),
  no.~4, 581--589.

\bibitem[J{\"a}g09a]{jager-bmm}
T.~J{\"a}ger, \emph{The concept of bounded mean motion for toral
  homeomorphisms}, Dyn. Syst. \textbf{24} (2009), no.~3, 277--297.

\bibitem[J{\"a}g09b]{jager-linearization}
\bysame, \emph{Linearization of conservative toral homeomorphisms}, Invent.
  Math. \textbf{176} (2009), no.~3, 601--616.

\bibitem[J{\"a}g11]{jager-elliptic}
\bysame, \emph{{Elliptic stars in a chaotic night}}, J. London Math. Soc.
  \textbf{84} (2011), no.~3, 595--611.

\bibitem[Jau11]{jaulent}
O.~Jaulent, \emph{Existence d'un feuilletage positivement transverse \`a un
  hom\'eomorphisme de surface}, (preprint), 2011.

\bibitem[KK09]{kk-spreading}
A.~Kocsard and A.~Koropecki, \emph{A mixing-like property and inexistence of
  invariant foliations for minimal diffeomorphisms of the 2-torus}, Proc. Amer.
  Math. Soc. \textbf{137} (2009), no.~10, 3379--3386.

\bibitem[Kor10]{koro}
A.~Koropecki, \emph{Aperiodic invariant continua for surface homeomorphisms},
  Math. Z. \textbf{266} (2010), no.~1, 229--236.

\bibitem[KT12a]{kt-irrotational}
A.~{Koropecki} and F.~A. {Tal}, \emph{{Area-preserving irrotational
  diffeomorphisms of the torus with sublinear diffusion}}, to appear in Proc.
  Amer. Math. Soc., eprint arXiv:1206.2409 (2012).

\bibitem[KT12b]{kt-pseudo}
A.~Koropecki and F.~A. Tal, \emph{Bounded and unbounded behavior for
  area-preserving rational pseudo-rotations}, eprint arXiv:1207.5573 (2012).

\bibitem[LC05]{lecalvez-equivariant}
P.~Le~Calvez, \emph{Une version feuillet\'ee \'equivariante du th\'eor\`eme de
  translation de {B}rouwer}, Publ. Math. Inst. Hautes \'Etudes Sci. (2005),
  no.~102, 1--98.

\bibitem[LC06]{lecalvez-hamiltonian}
\bysame, \emph{Periodic orbits of {H}amiltonian homeomorphisms of surfaces},
  Duke Math. J. \textbf{133} (2006), no.~1, 125--184.

\bibitem[LKFA90]{harper}
P.~Leb{\oe}uf, J.~Kurchan, M.~Feingold, and D.~P. Arovas, \emph{Phase-space
  localization: topological aspects of quantum chaos}, Phys. Rev. Lett.
  \textbf{65} (1990), no.~25, 3076--3079.

\bibitem[LM91]{llibre-mackay}
J.~Llibre and R.~S. MacKay, \emph{Rotation vectors and entropy for
  homeomorphisms of the torus isotopic to the identity}, Ergodic Theory Dynam.
  Systems \textbf{11} (1991), no.~1, 115--128.

\bibitem[Mos62]{moser-kam}
J.~Moser, \emph{On invariant curves of area-preserving mappings of an annulus},
  Nachr. Akad. Wiss. G\"ottingen Math.-Phys. Kl. II (1962), 1--20.

\bibitem[MZ89]{m-z}
M.~Misiurewicz and K.~Ziemian, \emph{Rotation sets for maps of tori}, Journal
  of the London Mathematical Society \textbf{40} (1989), no.~2, 490--506.

\bibitem[MZ91]{m-z2}
M.~Misiurewicz and K.~Ziemian, \emph{Rotation sets and ergodic measures for
  torus homeomorphisms}, Fund. Math. \textbf{137} (1991), no.~1, 45--52.

\bibitem[NS89]{stepanov}
V.~V. Nemitskii and V.~V. Stepanov, \emph{Qualitative theory of differential
  equations}, Courier Dover Publications, 1989.

\bibitem[PRK97]{pekarsky}
S.~Pekarsky and V.~Rom-Kedar, \emph{Uniform stochastic web in low-dimensional
  {H}amiltonian systems}, Phys. Lett. A \textbf{225} (1997), no.~4-6, 274--286.

\bibitem[Sol45]{solntzev}
G.~Solntzev, \emph{On the asymptotic behaviour of integral curves of a system
  of differential equations}, Bull. Acad. Sci. URSS. S\'er. Math. [Izvestia
  Akad. Nauk SSSR] \textbf{9} (1945), 233--240.

\bibitem[Whi33]{whitney}
H.~Whitney, \emph{Regular families of curves}, Ann. of Math. (2) \textbf{34}
  (1933), no.~2, 244--270.

\bibitem[Whi41]{whitney2}
\bysame, \emph{On regular families of curves}, Bull. Amer. Math. Soc.
  \textbf{47} (1941), 145--147.

\bibitem[ZZS{\etalchar{+}}86]{x_3}
G.~M. Zaslavski{\u\i}, M.~Y. Zakharov, R.~Z. Sagdeev, D.~A. Usikov, and A.~A.
  Chernikov, \emph{Stochastic web and diffusion of particles in a magnetic
  field}, Soviet Phys. JETP \textbf{64} (1986), no.~2, 294Ð--303, Translated
  from Zh. éksper. Teoret. Fiz. 91 (1986), no. 2, 500-516.

\end{thebibliography}

\end{document}